\documentclass[11pt]{amsart}

\usepackage[english]{babel} 

\usepackage[hyperref,x11names]{xcolor}
\usepackage{hyperref}
\hypersetup{%
colorlinks = true,       
breaklinks = true,       
urlcolor   = SteelBlue3, 
linkcolor  = Firebrick3, 
citecolor  = OliveDrab4, 
}

\usepackage[T1]{fontenc} 

\usepackage[utf8]{inputenc}

\usepackage[english=british]{csquotes}

\usepackage{amsmath} 
\usepackage{amsfonts}
\usepackage{mathrsfs}
\usepackage{amssymb}
\usepackage{amsthm}

\usepackage{microtype}
\newtheoremstyle{example}
  {}
  {}
  {}
  {\parindent}
  {\bfseries}
  {.}
  {\newline}
  {}

\newtheorem{theorem}{Theorem}[section]
\newtheorem{lemma}[theorem]{Lemma}
\newtheorem{proposition}[theorem]{Proposition}
\newtheorem{corollary}{Corollary}[theorem]
\theoremstyle{example}
\newtheorem{example}{Example}[section]
\theoremstyle{definition}
\newtheorem{definition}[theorem]{Definition}
\newtheorem{remark}[corollary]{Remark}

\newcommand*{\classicsets}[1]{\mathbb{#1}} 

\newcommand*{\nat}{\classicsets{N}}        

\newcommand*{\reals}{\classicsets{R}}

\newcommand*{\proj}{\classicsets{P}} 

\newcommand*{\from}{\mkern-3mu{\colon}\mkern-3mu\linebreak[0]}       
\renewcommand*{\to}{\myapplication{\varrightarrow}}

\DeclareMathOperator{\image}{Im}
\newcommand*{\im}[1]{\image(#1)}

\usepackage{tikz}
\usepackage{amsmath} 
\usepackage{amsfonts}
\usepackage{mathrsfs}
\usepackage{amssymb}
\usepackage{amsthm}

\usepackage{lscape}

\usepackage{paralist}

\usepackage{mathpazo}

\usepackage{mathrsfs}

\newcounter{parnum}[section]
\renewcommand{\theparnum}{\thesection.\arabic{parnum}}
\renewcommand{\paragraph}{\medskip\refstepcounter{parnum}\textbf{\theparnum}~\textbf}

\usepackage{multirow}

\DeclareMathOperator{\rk}{rk}

\DeclareMathOperator{\HF}{HF}
\DeclareMathOperator{\GL}{GL}
\DeclareMathOperator{\PGL}{PGL}

\newcommand{\C}{\mathbb{C}}
\newcommand*{\card}[1]{|#1|}
\newcommand*{\field}{\Bbbk}           

\usepackage[normalem]{ulem}

\newcommand{\mysetminusD}{\hbox{\tikz{\draw[line width=0.05em,line cap=round] (0.25em,0) -- (0,0.5em);}}}
\newcommand{\mysetminusT}{\mysetminusD}
\newcommand{\mysetminusS}{\hbox{\tikz{\draw[line width=0.035em,line cap=round] (0.175em,0) -- (0,0.3em);}}}
\newcommand{\mysetminusSS}{\hbox{\tikz{\draw[line width=0.033em,line cap=round] (0.125em,0) -- (0,0.25em);}}}

\newcommand{\mysetminus}{\mathbin{\mathchoice{\mysetminusD}{\mysetminusT}{\mysetminusS}{\mysetminusSS}}}

\newcommand*{\ppp}{(\proj S_1)^3\!\mysetminus\!\Delta}
\newcommand*{\ttt}{\Theta^3\mysetminus\Delta}

\DeclareMathOperator{\sym}{Sym}

\newcommand{\naturals}{\mathbb{N}}

\newcommand*{\cnm}[2]{\genfrac{[}{]}{0pt}{}{{\scriptstyle #1}}{{\scriptstyle #2}}}

\def\to{\longrightarrow}

\DeclareMathOperator{\rank}{rank}

\makeatletter
\def\RDerChar{\mathbf{R}}
\def\RDer{\@ifnextchar[{\R@Der}{\ensuremath{\RDerChar}}}
\def\R@Der[#1]{\ensuremath{\RDerChar^{#1}}}
\makeatother
\usepackage[normalem]{ulem}

\usepackage[capitalise]{cleveref}

\usepackage{filecontents}

\begin{document}

\setlength{\abovedisplayskip}{5pt} \setlength{\belowdisplayskip}{5pt}
\setlength{\abovedisplayshortskip}{5pt}\setlength{\belowdisplayshortskip}{5pt}

\begin{abstract}
  Waring problem for forms is important and classical in mathematics.
  It has been widely investigated because of its wide applications in several areas.
  In this paper, we consider the Waring problem for binary forms with complex coefficients.
  Firstly, we give an explicit formula for the Waring rank of any binary binomial and several examples to illustrating it.
  Secondly, we prove that, up to scalar multiplication, there are exactly $\binom{d-1}{2}$ binary forms of degree $d$ with Waring rank two and multiple of three fixed distinct linear forms.
\end{abstract}

\title[Binomials and dihedral cover]{On the Waring rank of binary forms: The binomial formula and a dihedral cover of rank two forms} 

\author[Brustenga]{Laura Brustenga i Moncusí}
\address{Universitat Autònoma de Barcelona, 08193 Bellaterra, Barcelona, Spain}
\email{brust@mat.uab.cat}

\author[Masuti]{Shreedevi K. Masuti}
\address{Chennai Mathematical Institute, H1-SIPCOT IT Park, Siruseri, Kelambakkam - 603 103, India}
\email{shreedevikm@cmi.ac.in}

\date{\today}
\thanks{SKM is supported by INSPIRE faculty award funded by Department of Science and Technology, Govt. of India. She is also partially supported by a grant from Infosys Foundation. She was supported by INdAM COFOUND Fellowships cofounded by Marie Curie actions, Italy, for her research in Genova during which part of the work is done and which also provided the travel support for participating in this school.\\
LBM is supported by DOGC11/05/2018/19006/3}



\maketitle


\section{Introduction}
\label{sec:introduction}
This paper concerns symmetric tensor decomposition as a sum of rank one tensors which is also known as the Waring problem for forms.
This topic has a rich history and recently received a huge interest mainly because of its wide applicability in areas as diverse as algebraic statistics, biology, quantum information theory, signal processing, data mining, machine 
learning, see \cite{comon-2008-symmet-tensor,kolda-2009-tensor-decom-applic,Lan12}. 
\medskip

Let $\field[x,y]$ be the standard graded polynomial ring with coefficients in the field \(\field\subseteq \C\).
For \(d\ge 0\), we denote by \(\field[x,y]_d\) the \(\field\)-vector space of forms of degree \(d\) in \(\field[x,y]\).
For every form \(F\in \field[x,y]_d\), by the classical Jordan Lemma (see \cite[Appendix III]{young-alg-of-invariants}),
there exist linear forms $L_1,\dots,L_r\in \field[x,y]_1$ and scalars \(a_1,\dots,a_r\in \field\) with \(r\le d+1\) such that
\[
  F=a_1L_1^d+\dots+a_rL_r^d.
\] 
When \(\field=\C\) (resp. \(\field=\reals\)), the least of such possible numbers $r$ is called the \emph{Waring rank of $F$} (resp. the \emph{real Waring rank of \(F\)}) and we denote it by \(\rk(F)\) (resp. \(\rk_{\reals}(F)\)).
Except for a few results comparing the Waring and the real Waring rank (see \cref{exm:real-vs-complex-1st,exm:real-vs-complex-2nd,prop:set-of-maximal-gap-real-complex-rank}), we will consider \(\field=\C\) throughout this paper.
\medskip

The longstanding problem of finding the Waring rank of a generic form in any number of variables has been solved recently by J.~Alexander and A.~Hirschowitz \cite{AH95}, after being remained open for more than one hundred years.
However, solving the Waring problem for generic form of degree $d$ does not give any information about specific form of degree $d$. 

There has been an intense research in finding the Waring rank of binary forms which goes back to the work of J. J. Sylvester. 
Sylvester \cite{sylvester-1851-essay-forms-sketch-eli-trans-can-forms,sylvester-1851-remarkable-dicovery-forms-and-hyperdeterminants} gave an explicit algorithm for computing the Waring rank of a binary form.
We refer to \cite{reznick-2012-length-binary-forms} for an excellent survey on the Waring problem for binary forms.
The Waring problem is even more interesting and challenging when the coefficient field is different from the field of complex numbers, see \cite{reznick-2012-length-binary-forms}.
Especially, because of the direct connection with the real world, there is a lot of interest in finding the decomposition of forms with real coefficients, see for instance \cite{BCG11}.
\medskip

As an immediate consequence of Sylverster's algorithm one can give an explicit formula for the Waring rank of monomials in $\C[x,y]$.
Moreover, an explicit formula for the Waring rank of monomials in any finite number of variables has been given recently in \cite[Proposition 3.1]{CCG12}.
This motives us to look beyond monomials for binary forms, so first we consider the binomial case and second forms with three distinct roots fixed.
\medskip

The first main result of this paper gives an explicit formula for the Waring rank of binomials in \(\C[x,y]\) (see \cref{thm:binomial}).
However, we remark that this is far from being an obvious generalization of the monomial case.
More generally, it is difficult to describe the Waring rank of $F_1+\cdots+F_k$ in terms of the Waring ranks of $F_1,\ldots,F_k$ as is evidenced by Strassen's conjecture (see \cite{carlini-catalisano-2018-e-computability,CCG12,CCO17,Tei15}). 

Our main tool in computing the Waring rank of a binomial form \(F\in\C[x,y]\) is Sylvester's algorithm (see Section \ref{SubSec:Syl}).
In order to apply this algorithm, we need to find a form of least degree in the apolar ideal \(F^{\perp}\) (see Section \ref{SubSec:Syl}).
For this, we give a nonzero form $g_1$ in \(F^{\perp}\) and, computing the Hilbert function (see Section \ref{SubSec:HF}) of \(T/F^{\perp}\) in certain degrees, we are able to conclude that $g_1$ is of least degree in \(F^{\perp}\).
Hence we avoid to compute the entire apolar ideal \(F^{\perp}\). 
\medskip

The second main goal is to study some geometric properties of the loci of binary forms with Waring rank two.
For a fixed degree \(d\ge 4\), we show that, up to scalar multiplication, there are \(\binom{d-1}{2}\) forms with Waring rank two and multiple of three distinct linear forms \(l_1,l_2,l_3\), or equivalently with three distinct roots fixed (see \cref{coro:finite-cardinal-of-Pi2,prop:cardinal-of-Pi2}).
For this, we build a dihedral cover of the set of forms with Waring rank two that are multiple of \(l_1l_2l_3\) (see \cref{def:The-Gamma-map}) and we show that there are \(\binom{d-1}{2}\) different orbits in the cover space by the dihedral action (see \cref{thr:Gamma-dihedral-cover}).
Such dihedral covers depend on the linear forms \(l_1,l_2,l_3\) and, ranging over all the three distinct linear forms, they cover the set of all binary forms with Waring rank two.

This cover also allows us to describe the set of all forms in \(\proj \reals[x,y]_d\) with a maximal gap between their Waring and real Waring rank (see \cref{prop:set-of-maximal-gap-real-complex-rank}).

Finally we offer a second proof of \cref{prop:cardinal-of-Pi2} based on more classical results.
Namely, the set of forms multiple of \(l_1l_2l_3\) form a linear subspace of \(\proj S_d\) which has codimension three (see \cref{rmk:linearity-of-pi}).
We show that it intersects transversely the second secant variety \(\overline{\Sigma}\) of the rational normal curve in \(\proj S_d\), which has dimension three (see \cref{rmk:2dn-sec-2-rank-forms}).
Then, the claim follows by Bézout's Theorem.
To this intent, we use Teracini's Lemma to parametrise the tangent space to a smooth point of \(\overline{\Sigma}\).
\medskip

The paper is organized as follows: We fix some notation and gather some preliminary results needed for our purpose in \cref{sec:preliminaries}. In Section \ref{Section:BinomialForm} we give an explicit formula for the Waring rank of a binomial form. \cref{ssec:two-variables} is devoted to study the dihedral cover and the geometric properties of the loci of binary forms with Waring rank two.

\section*{Acknowledgements}
This research is a Pragmatic\footnote{\url{http://www.dmi.unict.it/~pragmatic/docs/Pragmatic-main.html}} project.
We thank the speakers of Pragmatic 2017 for delivering the insightful lectures.
Specially to Enrico Carlini, who introduced us this fascinating subject and provided many useful ideas throughout the preparation of this manuscript.
We also thank the organisers of Pragmatic 2017 for providing the local support.
For the support on travel expenses, the first author thanks to MTM2016-75980-P, ``Métodos efectivos en Geometría Aritmética 2'' and the second author thanks to INdAM.

\section{Preliminaries}
\label{sec:preliminaries}

Let \(S\), \(T\) denote respectively the standard graded polynomial rings \(\C[x,y]\), \(\C[X,Y]\).
For \(i\ge 0\), denote respectively by \(S_i\), \(T_i\) the \(i\)-th graded component of \(S\), \(T\).

\paragraph{Apolarity theory and Sylvester's algorithm.}
\label{SubSec:Syl}
Consider the apolar action of \(T\) on \(S\), that is consider \(S\) as a \(T\)-module by means of the differentiation action
$$
\begin{array}{ cccc}
  \circ: & T \times S &\longrightarrow & S \\
         & (g , F) & \to & g \circ F = g(\partial_{X}, \partial_{Y})(F).
\end{array}
$$

\begin{definition}\label{def:apolar-ideal}
  Let \(F\) be a form in \(S_d\).
  A form \(g\in T_{d'}\) is called \emph{apolar to \(F\)} when \(g\circ F\in S_{d-d'}\) is the zero form.
  The \emph{apolar ideal to \(F\)}, denoted as $F^\perp$, is the homogeneous ideal in $T$ generated by all the forms apolar to \(F\), or equivalently
  $$
  F^\perp=\{g \in T\ |\ g \circ F = 0 \}\subseteq T.
  $$
\end{definition}

The so called Sylvester's algorithm below, an algorithm to compute Waring rank of binary forms, is a consequence of Sylvester's Theorem developed in \cite{sylvester-1851-essay-forms-sketch-eli-trans-can-forms,sylvester-1851-remarkable-dicovery-forms-and-hyperdeterminants}. 
Modern proofs of Sylvester's Theorem may be found in \cite[Theorem 2.1]{reznick-2012-length-binary-forms} (which is an elementary proof), \cite[Section 5]{kung-rota-1984-invariant-theory-binary-forms}, and with further discussion in \cite{kung-1986-gundelfinger-bin-forms,kung-1987-canonical-bin-forms-even-deg,kung-1990-canonical-bin-forms-theme-sylvester,reznick-1996-homogeneous-poly-to-pde}.
Here we just state the final version of the algorithm, see \cite[Remark 4.16]{carlini-grieve-oeding-4-lectures}, \cite[Algorithm 2]{bernardi-gimigliano-ida-2011-computing-symmetric-rank} or \cite[Section 3]{comas-seiguer-rank-binary}.

\begin{theorem}[Sylvester's algorithm]\label{thr:sylvesters-algorithm}
  Let \(F\) be a form in \(S_d\).
  By the Structure Theorem (see \cite[Theorem 1.44(iv)]{iarrobino-determinantal-loci}), the apolar ideal \(F^{\perp}\) to \(F\) can be generated by two forms \(g_1,g_2\in T\) with \(\deg(g_1)+\deg(g_2)=d+2\).
  Suppose that such forms \(g_1,g_2\in T\) are given with $\deg(g_1) \leq \deg(g_2)$.
  Then,
  \[
    \rk(F)=\begin{cases}
      \deg(g_1) & \mbox{ if } g_1 \mbox{ is square-free}\\
      \deg(g_2) & \mbox{ otherwise }.
    \end{cases}
  \]
\end{theorem}

\paragraph{Binary forms.} For the convenience of the reader, this section collects some well-known results on the Waring and the real Waring rank of binary forms that we will use.
\medskip

\cref{thr:maximal-possible-rank} below is proved in \cite[Theorems 4.9 and 4.10]{reznick-2012-length-binary-forms}.
For \(\field=\C\), it is an immediate consequence of Sylvester's Theorem.
More generally, \cref{thr:maximal-possible-rank} is a consequence of \cite[Corollary 2.2]{reznick-2012-length-binary-forms}, an adaptation of Sylvester's Theorem to forms with coefficients in a subfield \(\field\) of \(\C\). 

\begin{theorem}\label{thr:maximal-possible-rank}
  Let \(F\) be a form in \(\field[x,y]_d\) with \(d\ge 0\) and \(\field\) a subfiel of \(\C\). 
  Then, \(\rk_{\field}(F)\le d\).
\end{theorem}

For \(d\ge 3\), \cref{thr:maximal-rank-forms} below gives all the forms \(F\in S_d\) with maximal Waring rank, that is \(\rk(F)=d\).
Though it is well-known that \(\rk(x^{d-1}y)=d\), to the best of our knowledge (cf. \cite[Theorem 2.4]{tokcan-2017-warin-rank}), the converse has been proven only later \cite[Exercise 11.35]{Harris-AG-first-course} or \cite[Corollary 3]{bialynicki-birula-2010-extrem-binar-forms}.

\begin{theorem}\label{thr:maximal-rank-forms}
  Let \(F\in S_d\) with \(d\ge 3\).
  Then \(\rk(F)=d\) if and only if there are two distinct linear forms \(l_1,l_2\in S_1\) so that \(F=l_1^dl_2\).
\end{theorem}

\begin{definition}\label{def:hyperbolic-forms}
  A form \(F\in\reals[x,y]_d\) is called \emph{hyperbolic} if all its roots are real, that is it splits into linear factors over \(\reals\).
\end{definition}

Sylvester in \cite{sylvester-65-Newtons-hitherto-rule} already observed, by the dehomogenized case, that a hyperbolic binary form which is not a power of a linear form has real Waring rank \(d\), which is maximal.
See \cite[Section 3]{reznick-2012-length-binary-forms} for more history of this problem.
A partial converse of this statement for square free hyperbolic forms was established in \cite{causa-re-2011-maximum-real-rank,comon-ottaviani-2012-typical-real-rank}
and the full converse is proved in \cite[Theorem 2.2]{blekherman-2016-real-rank}.

\begin{theorem}\label{thr:hyperbolic-max-real-rank}
  Let \(F\) be a form in \(\reals[x,y]_d\) with \(d\ge 3\).
  Suppose that \(F\) is not a \(d\)-th power of a linear form.
  Then, the real Waring rank of \(F\) is \(d\) if and only if \(F\) is hyperbolic.
\end{theorem}

\paragraph{Hilbert function.} 
\label{SubSec:HF}
Let $I$ be a homogeneous ideal in $T$.
The \emph{Hilbert function of $T/I$} is an important numerical invariant associated to $T/I$ defined as
\[
  \HF_{T/I}(i)=\dim_\C T_i - \dim_\C I_i,
\]
where $I_i$ denotes the $\C$-vector space of degree \(i\) forms in $I$.
In Section \ref{Section:BinomialForm}, our strategy is to compute the Hilbert function of $T/F^{\perp}$ and extract information about the Waring rank of $F$ from it using Sylvester's algorithm.

For $F\in S_d$, let $\langle F\rangle$ denote the $T$-submodule of $S$ generated by $F$, that is, the $\C$-vector space generated by $F$ and by the corresponding derivatives of all orders.
It is well-known that $\langle F \rangle$ determines the Hilbert function of $T/F^\perp$ as follows (see \cite{Ger96}):
\begin{equation}
  \label{Eqn:HFViaDual}\tag{\dag}
  \HF_{T/F^\perp}(i)= \dim_\C \langle F \rangle_i
\end{equation}
where $\langle F \rangle_i$ denotes the $\C$-vector space generated by the forms of degree $i$ in $\langle F \rangle$.

\paragraph{Group actions.} We review some basic definitions and facts about a group acting on a set.
Actually, we consier groups acting on a topological space but, since we will apply this definitions to finite topological spaces with the discrete topology, we may consider them as simply sets.

\begin{definition}\label{def:finer-eq-partition}
  Let \(X\) be a set.
  Let \(\Omega\) and \(\Omega'\) be two partitions of \(X\).
  The partition \(\Omega\) is \emph{finer than} \(\Omega'\) if for every \(\omega\in \Omega\) there is \(\omega'\in \Omega'\) such that \(\omega\subseteq \omega'\).
\end{definition}

Recall that to two partitions \(\Omega\) and \(\Omega'\) of a set \(X\) are equal if and only if they are mutually finer.

Given a map of sets \(f\from X\to Y\), we denote by \(X/f\) the partition of \(X\) determined by the fibres of \(f\).

\begin{definition}\label{def:group-theoretical-kernel}
  Let \(X\) be a set and \(G\) a group acting on it.
  Given \(x\in X\) the \emph{orbit of \(x\) by the action of \(G\)} is the set \(\{g\cdot x\}_{g\in G}\) and it is denoted by \(G\cdot x\).
  The set of all orbits is denoted by \(X/G\), which is a partition of \(X\).
\end{definition}

Recall that a group \(G\) \emph{acts faithfully on} a set \(X\) if there are no nontrivial \(g\in G\) such that \(g\cdot x=x\) for all \(x\in X\).

\begin{definition}\label{def:G-cover}
  Let \(X\), \(Y\) be finite sets and \(G\) a group.
  A map \(f\from X\to Y\) is a \emph{\(G\)-cover of \(Y\)} if the group \(G\) acts faithfully on \(X\) and the partitions \(X/f\) and \(X/G\) are equal.
\end{definition}

\section{Waring rank of binomial forms}
\label{Section:BinomialForm}
The main goal of this section is to give an explicit formula for the Waring rank of any binomial form (\cref{thm:binomial}).
In particular, we prove that the Waring rank of a binomial form does not depend on its coefficients.
We give several examples, \cref{Examples:thm} illustrates the result itself and
\cref{Example:Trinomial,exm:real-vs-complex-2nd} that the Waring rank of a trinomial and the real Waring rank of a binomial may depend on its coefficients respectively.

Recall that, given a nonzero homogeneous ideal $I$ of $T$, the \emph{initial degree} of \(I\) is the least integer \(i>0\) for which \(I_i\) is not zero.

\begin{theorem}\label{thm:binomial}
  Let $F\in S_d$ be a binomial form.
  That is, \(F=ax^ry^{s+\alpha}+ bx^{r+\alpha}y^s\) for some \(a,b\in\C\setminus \{0\}\) and \(r,s,\alpha\in \naturals\).
  Assume (without loss of generality) that $0 \leq r \leq s$ and $\alpha \geq 1$.
  Let \(q,j\) be the unique nonnegative integers such that $r =q\alpha+ j$ with \(0\le j < \alpha\) and set \(\delta=r+\alpha-s\).
  Then the {Waring rank of $F$ can be computed using the following table.} 
  \begin{center}
    \begin{tabular}{|l l|c|}
      \hline
      \multicolumn{2}{|c|}{Conditions} & \(\rk(F)\) \\
      \hline
      (1)~~ \(\delta \leq 0\) & & \(s+1\)\\
      \hline
      \multirow{4}{8em}{(2)~~ \(\delta > 0\)}
                                       & \(j=0\), \(r=s\) and \(\alpha> 1\) & \(s+2\) \\
                                       & \(j=\delta\) & \(s+1\) \\
                                       & \(j>\delta\) & \(r+\alpha+1\) \\
                                       & otherwise & \(r+\alpha-j\) \\
      \hline
    \end{tabular}
  \end{center}
\end{theorem}
\begin{proof}
  Since $\rk(F)$ is invariant by a linear change of coordinates, we may assume that $a=\frac{(r+\alpha)!s!}{r!(s+\alpha)!}$ and $b=1$. For every pair of integers \(m,n\),
  we set $\cnm{m}{n}$ equal to $\frac{m!}{(m-n)!}$ if $m \geq n \geq 0$ and equal to zero otherwise.

  \begin{asparaenum}[$(1)$]
  \item Suppose \(\delta\le 0\), that is $s \geq r+\alpha$.
    Clearly, $g_1=X^{r+\alpha+1} \in F^\perp$.
    We claim that the initial degree of $F^\perp$ is $r+\alpha+1$.
    For $0 \leq i \leq s, $ we have
    \begin{eqnarray*}
      X^i Y^{s-i} \circ F
      =\textstyle a \cdot \cnm{r}{i} \cdot \cnm{s+\alpha}{s-i} \cdot x^{r-i}y^{\alpha+i}+
      \cnm{r+\alpha}{i} \cdot \cnm{s}{s-i} \cdot x^{r+\alpha-i}y^{i} .
    \end{eqnarray*}
    It is easy to see that the set $\{X^i Y^{s-i} \circ F:0 \leq i \leq r+\alpha\} \subseteq \langle F \rangle_{r+\alpha}$ is $\C$-linearly independent.
    Therefore by \eqref{Eqn:HFViaDual},
    $$
    \HF_{T/F^\perp}(r+\alpha)=\dim_\C \langle F \rangle_{r+\alpha} = r+\alpha+1.
    $$
    This implies that the nonzero homogeneous elements of $F^\perp$ have degree at least $r+\alpha+1$.
    Since $g_1 \in F^\perp$, the initial degree of $F^\perp$ is $r+\alpha+1$.

    Therefore $g_1$ is part of a minimal generating set of $F^\perp$ and hence there exists $g_2 \in T_{s+1}$ such that $F^\perp=(g_1,g_2)$. As $s \geq r+\alpha$ and $g_1$ is not square free, by Sylvester's algorithm
    \[
      \rk(F) = s+1.
    \]

  \item Assume \(\delta>0\), that is $s<r+\alpha$.
    In order find the Waring rank of $F$ in this case, we compute the Hilbert function of $T/F^\perp$ and use this information to compute the element of smallest degree in $F^\perp.$ Since this computation depends on $j,$ we divide the proof in the following four cases:(i) $ 0\leq j < \lceil \frac{\delta-1}{2}\rceil;$ (ii) $\delta$ is odd and $j=\frac{\delta-1}{2};$ (iii) $\lceil \frac{\delta-1}{2} \rceil < j <\delta -1$, OR $\delta$ is even and $j=\frac{\delta}{2};$ and (iv) $\delta -1 \leq j$.
    Notice that the first row of (2) in the table is covered in Case (i), the second and the third rows of (2) of the table are covered in Case (iv) whereas the last row of the table is covered in Cases (i),(ii),(iii) and (iv).
    \medskip

    \noindent{\bf Case i:~} $ 0\leq j < \lceil \frac{\delta-1}{2}\rceil-1$
    First we show that the initial degree of $F^\perp$ is $s+j+2$.
    For this it suffices to show that $\HF_{T/F^\perp}(s+j+1)=s+j+2$ and that there exists a nonzero form of degree $s+j+2$ in $F^\perp$.
    For \(0\le i\le r+\alpha-j-1\), we have
    \begin{align}
      \label{Eqn:Case1}
      & X^{r+\alpha-j-1-i}Y^i \circ F \\
      & \textstyle = a\cdot\cnm{r}{r+\alpha-j-1-i} \cdot \cnm{s+\alpha}{i} \cdot x^{j+1+i-\alpha}y^{s+\alpha-i} + \cnm{r+\alpha}{r+\alpha-j-1-i} \cdot \cnm{s}{i} \cdot
        x^{j+1+i}y^{s-i}. \nonumber
    \end{align}

    By \eqref{Eqn:Case1} for $i=0,1, \ldots, \min\{s,\alpha-j-2\}$, we get
    $$\{x^{j+1}y^s,x^{j+2}y^{s-1},\ldots,x^{j+1+\min\{s,\alpha-j-2\}}y^{s-\min\{s,\alpha-j-2\}}\} \subseteq \langle F\rangle_{s+j+1}.$$
    \noindent
    By assumption on $j$, we have $2(j+1) \leq {\delta}$ and hence $ (q+1) \alpha-(j+1) > s$.
    Therefore taking $i= (q+1) \alpha-(j+1), (q+1) \alpha-j,\ldots, (q+1) \alpha-1$ in \eqref{Eqn:Case1}, we get
    \[
      \{x^{q\alpha}y^{s+j+1-q\alpha},x^{q\alpha+1}y^{s+j-q\alpha},\ldots,x^{q\alpha+j}y^{s-q\alpha+1}\} \subseteq \langle F\rangle_{s+j+1}.
    \]
    Now substituting $i=\alpha-j-1, \alpha-j,\ldots,(q+1) \alpha-(j+2)$ in \eqref{Eqn:Case1}, we conclude that all monomials of degree $s+j+1$ belong to $\langle F\rangle$. Therefore by \eqref{Eqn:HFViaDual},
    $\HF_{T/F^\perp}(s+j+1)=s+j+2$. Hence the nonzero homogeneous elements of $F^\perp $ have degree at least $s+j+2$.

    \noindent We claim that
    \[
      g_1=\sum_{i=0}^{q} (-1)^i X^{r+1- i\alpha}Y^{i \alpha+s-r+j+1} \in (F^\perp)_{s+j+2}.
    \]
    If $q=0$, then $r=j$. Therefore $g_1=X^{r+1}Y^{s+1}$ which clearly belongs to $F^\perp$.
    Hence assume that $q >0$.
    Then
    \begin{align*}
      g_1\circ F
      & =0\,+ \cnm{r+\alpha}{r+1}\cdot \cnm{s}{s-r+j+1} \cdot x^{\alpha-1}y^{r-j-1}+\\
      &\quad +\sum_{i=1}^{q-1}(-1)^i \Biggl( a\cdot \cnm{r}{r+1-i\alpha}\cdot \cnm{s+\alpha}{i\alpha+s-r+j+1}\cdot x^{i\alpha-1}y^{(1-i)\alpha+r-j-1}+ \\
      &\qquad\qquad\qquad\ +\cnm{r+\alpha}{r+1-i\alpha} \cdot
        \cnm{s}{i\alpha+s-r+j+1} \cdot x^{(i+1)\alpha-1}y^{-i\alpha+r-j-1} \Biggr)\\
      &\quad +(-1)^q a \cdot \cnm{r}{r+1-q\alpha} \cdot \cnm{s+\alpha}{q\alpha+s-r+j+1} \cdot x^{q\alpha-1}y^{(1-q)\alpha+r-j-1}+0\\
      &=0.\\
    \end{align*}
    Thus the initial degree of $F^\perp$ is $s+j+2$. Therefore $F^\perp=(g_1,g_2)$ for some $0 \neq g_2 \in T_{r+\alpha-j}$. Moreover,
    \begin{eqnarray*}
      g_1 = \begin{cases}
        Y^2 \left( \sum_{i=0}^{q} (-1)^i X^{r+1- i\alpha}Y^{i \alpha+s-r+j-1} \right) & \mbox{ if } j \geq 1 \mbox{ OR } s -r >0\\
        Y \left( \sum_{i=0}^{q} (-1)^i X^{r+1- i\alpha}Y^{i \alpha} \right) & \mbox{ if } j=0 \mbox{ and } r=s.
      \end{cases}
    \end{eqnarray*}
    Suppose that $j=0$ and $r=s$.
    Then
    $$
    \left( \sum_{i=0}^{q} (-1)^i X^{r+1- i\alpha}Y^{i \alpha} \right) (X^\alpha+Y^\alpha)=X^{r+\alpha+1}+(-1)^{q}X Y^{r+\alpha}.
    $$
    Since $X^{r+\alpha+1}+(-1)^{q} X Y^{r+\alpha}$ is square-free, $g_1$ is square-free if $j=0$ and $r=s$. Clearly, if $j \geq 1$, OR $s-r >0$, then
    $g_1$ is not square-free. Therefore by Sylvester's algorithm
    \begin{eqnarray*}
      \rk(F) = \begin{cases}
        r+\alpha-j & \mbox{ if } j \geq 1 \mbox{ OR } s-r >0\\
        s+2 & \mbox{ if } j=0 \mbox{ and } r=s.\\
      \end{cases}
    \end{eqnarray*}

    \noindent{\bf Case ii:~} $\delta$ is odd and $j=\frac{\delta-1}{2}$

    First we prove that the initial degree of $F^\perp$ is $s+j+1$. For $0 \leq i \leq r+\alpha-j$,
    we have
    \begin{align}
      \label{Eqn:Case2}
      & X^{r+\alpha-j-i}Y^i \circ F \\
      &= a\cdot\textstyle \cnm{r}{r+\alpha-j-i} \cdot \cnm{s+\alpha}{i} \cdot x^{j+i-\alpha}y^{s+\alpha-i} +
        \cnm{r+\alpha}{r+\alpha-j-i} \cdot \cnm{s}{i} \cdot x^{j+i}y^{s-i} .\nonumber
    \end{align}

    \noindent
    Substituting $i=0,1,\ldots,\alpha-j-1$ in \eqref{Eqn:Case2}, we get
    $$\{x^jy^s,x^{j+1}y^{s-1},\ldots,x^{\alpha-1}y^{s-\alpha+j+1}\} \subseteq \langle F\rangle_{s+j}.$$
    As $j+1=\delta-j=r+\alpha-s-j$, we have $j+s+1-\alpha=r-j=q \alpha$. Therefore taking $i=s+1,s+2,\ldots,r+\alpha-j-1$ in \eqref{Eqn:Case2} we get
    \[
      \{x^{q\alpha}y^{s-q\alpha+j},x^{q\alpha+1}y^{s-q\alpha+j-1},\ldots,x^{q\alpha+j-1}y^{s-q\alpha+1}\} \subseteq \langle F \rangle_{s+j}.
    \] Now substituting $i=\alpha-j,\alpha-j+1,\ldots,s$ in \eqref{Eqn:Case2} we conclude that all the monomials of degree $s+j$ are in $\langle F \rangle$. Hence by \eqref{Eqn:HFViaDual}, $\HF_{T/F^\perp}(s+j)=s+j+1$.
    Therefore the nonzero homogeneous elements of $F^\perp$ have degree at least $s+j+1$.

    We claim that
    \[
      g_1=\sum_{i=0}^{q+1} (-1)^i X^{s+j+1-i \alpha} Y^{i \alpha} \in (F^\perp)_{s+j+1}.
    \]
    Indeed,
    \begin{align*}
      g_1\circ F
      & =0\,+\cnm{r+\alpha}{s+j+1} \cdot \cnm{s}{0} \cdot x^{r+\alpha-s-j-1}y^s \\
      &\quad + \sum_{i=1}^q (-1)^i \biggl(a\cdot \cnm{r}{s+j+1-i \alpha} \cdot \cnm{s+\alpha}{i \alpha} \cdot x^{r+i\alpha-s-j-1}y^{s+(1-i)\alpha}+\\
      &\quad\qquad\qquad\qquad+\cnm{r+\alpha}{s+j+1-i \alpha} \cdot \cnm{s}{i \alpha} \cdot x^{r+(i+1)\alpha-s-j-1}y^{s-i\alpha} \biggr)\\
      &\quad + (-1)^{q+1} a \cdot \cnm{r}{0} \cdot \cnm{s+\alpha}{(q+1)\alpha} \cdot x^r y^{s-q \alpha}+0\\
      &= 0.
    \end{align*}
    Therefore there exists $0\neq g_2\in T_{r+\alpha-j+1}$ such that $F^\perp=(g_1,g_2)$.
    As
    \[
      g_1 \left (X^{\alpha} + Y^\alpha \right) = X^{s+j+1+\alpha}+(-1)^{q+1} Y^{s+j+1+\alpha},
    \]
    and $X^{s+j+1+\alpha}+(-1)^{q+1} Y^{s+j+1+\alpha}$ is square-free, $g_1$ is also square-free. Hence by Sylvester's algorithm $\rk(F)=s+j+1=r+\alpha-j$.

    \noindent{\bf Case iii:~}$ \lceil \frac{\delta-1}{2}\rceil < j < \delta - 1 $ OR $\delta$ is even and $j=\frac{\delta}{2}$

    Let $k=\delta-j$. Then $ 2\leq k \leq \lceil \frac{\delta-1}{2}\rceil$. We show that $\rk(F)=r+\alpha-j=s+k$. For this it suffices to show that $\HF_{T/F^\perp}(s+k-1)=s+k$ and that there exists a square-free polynomial of degree $s+k$ in $F^\perp$.
    For $0 \leq i \leq r+\alpha-k+1$, we have
    \begin{align}
      \label{Eqn:Case3}
      & X^{r+\alpha-k+1-i}Y^i \circ F\\
      &\textstyle=a\cdot\cnm{r}{r+\alpha-k+1-i} \cdot \cnm{s+\alpha}{i} \cdot x^{k-1+i-\alpha}y^{s+\alpha-i} +
        \cnm{r+\alpha}{r+\alpha-k+1-i} \cdot \cnm{s}{i} \cdot x^{k-1+i}y^{s-i}.\nonumber
    \end{align}
    Substituting $i=0,1,\ldots,\alpha-k$ in \eqref{Eqn:Case3} we get
    \[
      \{x^{k-1}y^{s},x^{k}y^{s-1},\ldots,x^{\alpha-1}y^{s-\alpha+k}\} \subseteq \langle F\rangle_{s+k-1}.
    \]
    Notice that $r=q\alpha+j=q\alpha+(\delta-k)$.

    Taking $i=s+1,s+2,\ldots,(q+1)\alpha-1$ in \eqref{Eqn:Case3} we get

    \[
      \{x^{q \alpha}y^{s-q\alpha+k-1},x^{q \alpha+1}y^{s-q\alpha+k-2},\ldots,x^{q \alpha+k-2}y^{s-q \alpha+1}\} \subseteq \langle F \rangle_{s+k-1}.
    \]
    (Note that as $k \leq \lceil \frac{\delta-1}{2}\rceil$, $k \leq \delta-k $ and hence $(q+1)\alpha-1 \leq r+\alpha-k+1$.)

    \noindent Now using \eqref{Eqn:Case3} for $\alpha-k < i \leq s$, we conclude that all the monomials of degree $s+k-1$ are in $ \langle F \rangle$, and thus $\langle F \rangle_{s+k-1}=S_{s+k-1}$. Hence by \eqref{Eqn:HFViaDual}
    $\HF_{T/F^\perp}(s+k-1)=s+k$.

    \noindent Next we claim that
    \[
      g_1=\sum_{i=0}^{q+1} (-1)^iX^{s+k-i\alpha}Y^{i\alpha} \in (F^\perp)_{s+k}.
    \]
    We have
    \begin{align*}
      g_1 \circ F & = 0\, + \cnm{r+\alpha}{s+k} \cdot \cnm{s}{0} \cdot x^{r+\alpha-s-k} y^{s} \\
                  &\quad + \sum_{i=1}^q (-1)^i\biggl(a \cdot \cnm{r}{s+k-i\alpha}\cdot \cnm{s+\alpha}{i\alpha} \cdot x^{r-s-k+i \alpha}y^{s+(1-i)\alpha} +\\
                  &\quad\qquad\qquad\qquad+ \cnm{r+\alpha}{s+k-i\alpha} \cdot \cnm{s}{i\alpha} \cdot x^{r-s-k+(i+1)\alpha}y^{s-i\alpha}\biggr)\\
                  &\quad + (-1)^{q+1} a \cdot \cnm{r}{0} \cdot \cnm{s+\alpha}{(q+1)\alpha} \cdot x^r y^{s-q\alpha}+0\\
                  &=0.\\
    \end{align*}
    (Notice that $(q+1) \alpha=q \alpha + \alpha=s-\alpha+k+\alpha=s+k>s$). Hence $g_1 \in (F^\perp)_{s+k}$. Also
    \[
      g_1 (X^\alpha+ Y^\alpha)=X^{s+k+\alpha}+(-1)^{q+1} Y^{s+k+\alpha}.
    \]
    As $X^{s+k+\alpha}+(-1)^{q+1} Y^{s+k+\alpha}$ is square-free, $g_1$ is also square-free.
    Therefore by Sylvester's algorithm $\rk(F)=s+k=s+\delta-j=r+\alpha-j$.

    \noindent {\bf Case iv:~} $\delta-1 \leq j \leq \alpha-1$

    First we show that the initial degree of $F^\perp$ is $s+1$. For this it suffices to show that $\HF_{T/F^\perp}(s)=s+1$ and that there exists a nonzero form of degree $s+1$ in $F^\perp$. For $0 \leq i \leq r+\alpha$, we have
    \begin{align*}
      & X^{r+\alpha-i}Y^i \circ F \\
      &\textstyle =a\cdot\cnm{r}{r+\alpha-i} \cdot \cnm{s+\alpha}{i} \cdot x^{i - \alpha} y^{s+\alpha-i} + \cnm{r+\alpha}{r+\alpha-i} \cdot \cnm{s}{i} \cdot x^{i} y^{s-i}. \nonumber
    \end{align*}

    \noindent Therefore
    \begin{align*}
      & \{x^iy^{s-i}:0 \leq i < \alpha\}\\
      & \cup\textstyle\left\{a\cdot\cnm{r}{r+\alpha-i} \cdot \cnm{s+\alpha}{i} \cdot x^{i - \alpha} y^{s+\alpha-i} + \cnm{r+\alpha}{r+\alpha-i} \cdot \cnm{s}{i} \cdot x^{i} y^{s-i}:\alpha \leq i \leq s\right\} \subseteq \langle F \rangle_s.
    \end{align*}
    This implies that $\langle F \rangle_s=S_{s}$. Hence by \eqref{Eqn:HFViaDual},
    \[
      \HF_{T/F^\perp}(s)=s+1.
    \]
    We claim that
    \[
      g_1=\sum_{i=0}^{q+1}(-1)^i X^{s-(j-\delta)-i\alpha}Y^{i\alpha+(j-\delta)+1} \in (F^\perp)_{s+1}.
    \]
    We have
    \begin{align*}
      g_1\circ F
      & = 0\,+ \cnm{r+\alpha}{s-j+\delta }\cdot \cnm{s}{j-\delta+1} \cdot x^{r+\alpha-s+j-\delta}y^{s-j+\delta-1} \\
      &\quad +\sum_{i=1}^q(-1)^i \biggl(a \cdot \cnm{r}{s-j+\delta-i \alpha}\cdot \cnm{s+\alpha}{i\alpha+j-\delta+1} \cdot x^{r-s+j-\delta+i\alpha}y^{s-j+\delta+(1-i)\alpha-1} \\
      &\qquad\qquad\qquad + \cnm{r+\alpha}{s-j+\delta-i \alpha} \cdot \cnm{s}{i\alpha+j-\delta+1} \cdot x^{r-s+j-\delta+(i+1)\alpha}y^{s-j+\delta-i\alpha-1}\biggr) \\
      & \quad + (-1)^{q+1} a \cdot \cnm{r}{0} \cdot \cnm{s+\alpha}{(q+1)\alpha+(j-\delta+1)} \cdot x^r y^{s-q\alpha-(j-\delta+1)}+0 \\
      &=0.
    \end{align*}
    (Notice that $(q+1) \alpha+j-\delta=r+\alpha-\delta=s$, hence $s-(j-\delta)=(q+1) \alpha=r+(\alpha-j) > r$.)
    Hence $g_1 \in (F^\perp)_{s+1}$.
    Therefore the initial degree of $F^\perp$ is $s+1$. Hence $g_1$ is part of a minimal generating set of $F^\perp$. Therefore there exists $0 \neq g_2 \in T_{r+\alpha+1}$ such that $F^\perp=(g_1,g_2)$. Clearly, if $j \geq \delta+1$, then $g_1$ is not square-free.
    If $\delta-1 \leq j \leq \delta$, then
    \begin{eqnarray*}
      g_1 (X^\alpha+ Y^\alpha) =\begin{cases}
        X^{s+\alpha+1}+(-1)^{q+1} Y^{s+\alpha+1} & \mbox{ if }j= \delta-1\\
        Y (X^{s+\alpha}+(-1)^{q+1} Y^{s+\alpha}) & \mbox{ if }j=\delta,
      \end{cases}
    \end{eqnarray*}
    which implies that $g_1$ is square-free. Hence by Sylvester's algorithm
    \[
      \rk(F)=\begin{cases}
        s+1 & \mbox{ if } \delta-1 \leq j \leq \delta\\
        r+\alpha+1 & \mbox{ if } j \geq \delta+1.
      \end{cases}\qedhere
    \]
  \end{asparaenum}
\end{proof}

\begin{remark}
  \label{Remark:GenSpeRank}
  The generic rank of a form of degree \(r+s+\alpha\) is $\lceil \frac{r+s+\alpha+1}{2} \rceil$, see \cite{AH95} or \cite{carlini-grieve-oeding-4-lectures}.
  \cref{thm:binomial} illustrates that the Waring rank of a specific a form behaves as weirdly as possible compared to the generic rank.
  In particular, the Waring rank of a specific form can be smaller or larger than the generic rank.
\end{remark}

We illustrate \cref{thm:binomial} in the following example.

\begin{example}
  \label{Examples:thm}
  Let the notation be as in Theorem \ref{thm:binomial}.
  \begin{asparaenum}[(1)]
  \item [(0)] ($r=s=0$ and $\alpha \geq 1$) Let $F=x^\alpha+y^{\alpha}$. Clearly,
    $$
    \rk(F)=\begin{cases}
      1 & \mbox{ if } \alpha=1\\
      2 & \mbox{ if } \alpha=2.
    \end{cases}
    $$

    On the other hand, since $r=s=0, $ we have $\delta=\alpha \geq 1$ and $j=0$. Therefore $\rk(F)$ coincides with the Waring rank stated in Theorem \ref{thm:binomial}.
  \item [(1)] ($r=s$ and $\alpha=1$) Let $F=x^ry^{r+1}+x^{r+1}y^r$. In this case $F^\perp=(g_1,g_2)$ where
    \[
      g_1=X^{r+1}-X^rY+\cdots+(-1)^iX^{r+1-i}Y^i+(-1)^{r+1} Y^{r+1} \mbox{ and } g_2=X^{r+2}.
    \]
    Since $g_1$ is square-free, by Sylvester's algorithm $\rk(F)=r+1$. 

    On the other hand, since $r=s$, we have $\delta=r+\alpha-s=\alpha=1$ and hence $j=0$ for every nonnegative integer $r$. Therefore $ \rk(F)=r+1$ by Theorem \ref{thm:binomial} also. 

  \item [(2)] ($r=0,~s>0$ and $\alpha \geq 1$). Let $F=ay^{s+\alpha}+x^{\alpha}y^s$ where $a=\frac{\alpha!s!}{(s+\alpha)!}$. In this case $\delta=r+\alpha-s=\alpha-s$ and $j=0 $ if $\delta >0$. Hence, by Theorem \ref{thm:binomial},
    $$\rk(F)=\begin{cases}
      s+1 & \mbox{ if } \alpha \leq s \\
      \alpha & \mbox{ if } \alpha >s.
    \end{cases}
    $$ 

    This can be verified directly (without using Theorem \ref{thm:binomial}) as follows: We have 
    \[
      F^{\perp}=\begin{cases}
        (X^{\alpha+1},Y^{s+1} - X^\alpha Y^{s-\alpha+1}) & \mbox{ if } \alpha \leq s \\
        (XY^{s+1},X^{\alpha}-Y^{\alpha}) & \mbox{ if } \alpha >s
      \end{cases}
    \]
    and hence by Sylvester's algorithm $\rk(F)$ is as required.
  \end{asparaenum}
\end{example}

The following example illustrates that the Waring rank of a trinomial may depend on its coefficients.

\begin{example}
  \label{Example:Trinomial}
  Consider the quadratic form $F=x^2+xy+y^2$. Then
  \[
    F=[x \hspace*{2mm} y]A_F [x \hspace*{2mm} y]^T \mbox{ where }
    A_F=\left(
      \begin{array}{cc}
        1 & 1/2\\
        1/2 & 1
      \end{array}
    \right).
  \] It is a standard fact from linear algebra that $\rk(F)=\rank(A_F)$.
  Hence \(\rk(F)=2\).
  On the other hand, if $G=x^2+2xy+y^2=(x+y)^2$, then $\rk(G)=1$.
  This shows that unlike the binomial case, even in the binary case the Waring rank of a trinomial may depend on its coefficients.
\end{example}

Now we compare the Waring and the real Waring rank of a binomial form. 

\begin{example}
  \label{exm:real-vs-complex-1st}
  Let $F=x^ry^r(x+y)$ where $r \geq 1$.
  By Theorem \ref{thm:binomial} $\rk(F)=r+1$.
  Whereas, since $F$ splits completely into linear factors in $\mathbb{R}$, $\rk_{\mathbb{R}}(F)=2r+1$ by \cref{thr:hyperbolic-max-real-rank}.
\end{example}

Moreover, unlike in the complex case, the real Waring rank of a binomial form may depend on its coefficients .

\begin{example}
  \label{exm:real-vs-complex-2nd}

  By \cref{thr:hyperbolic-max-real-rank}, $\rk_{\mathbb{R}}( x^3-xy^2)=3$ whereas $\rk_{\mathbb{R}}(x^3+xy^2)=2$. But $\rk(x^3 \pm xy^2)=\rk(y^3 \pm x^2y)=2$ by Theorem \ref{thm:binomial} (take $r=0,~s=1$ and $\alpha=2$ in Theorem \ref{thm:binomial}). 
\end{example}

\section{On binary forms with Waring rank two}
\label{ssec:two-variables}
This section studies some geometric properties of the loci of binary forms with Waring rank two.
Fix an integer \(d\ge 4\).
We show that a form with Waring rank two is square free (see \cref{coro:sq-freeness-of-rk-2,rmk:any-number-of-variables}).
Moreover, for every three distinct linear forms \(l_1,l_2,l_3\in\proj S_1\) there are just a finite number of forms in \(\proj S_d\) with Waring rank two and multiple of \(l_1l_2l_3\) (see \cref{coro:finite-cardinal-of-Pi2}).
This number does not depend on the linear forms \(l_1,l_2,l_3\) and it is \(\binom{d-1}{2}\)
(see \cref{prop:cardinal-of-Pi2}).
\medskip

We build a dihedral cover of the set of forms with Waring rank two and multiples of \(l_1l_2l_3\) (see \cref{def:The-Gamma-map,thr:Gamma-dihedral-cover}, see also \cref{def:G-cover}) and show that there are \(\binom{d-1}{2}\) different orbits in the cover space by the dihedral action (see \cref{prop:cardinal-of-Pi2}).
This cover allows us also to describe the set of all forms in \(\proj \reals[x,y]_d\) with a maximal gap between their Waring and real Waring rank (see \cref{prop:set-of-maximal-gap-real-complex-rank}).
Finally we offer a second proof of \cref{prop:cardinal-of-Pi2} based on more classical results.
Namely, the set of forms multiple of \(l_1l_2l_3\) form a linear subspace of \(\proj S_d\) (see \cref{rmk:linearity-of-pi}).
We show that it intersects transversely the second secant variety of the rational normal curve in \(\proj S_d\) and that their dimensions are complementary, then result follows by Bezout's theorem.
\medskip

We denote by \(\nat_d\) the set of integers \(l\) with \(0\le l\le d-1\).
Let \(X\) be a set.
We denote by \(\card{X}\) the cardinal of \(X\) and by \(\Delta(X)\) (or simply by \(\Delta\) when there is no risk of confusion) the big diagonal in \(X^3\), that is \(X^3\!\mysetminus\!\Delta\) is the set of triples of pairwise distinct elements of \(X\).

We consider the \(\C\)-vector space \(S_1\) with basis \(\{x,y\}\).
This fixes an action of the general linear group \(\GL_2(\C)\), or simply \(\GL_2\), on \(S_1\).
Namely, every \(\alpha\in \GL_2\) determines a unique invertible \(\C\)-linear map \(\alpha\from S_1\to S_1\) and then \(\alpha\cdot l=\alpha(l)\) for all \(l\in S_1\).
Since \(S_d=\sym^d S_1\), this action extends to an action of \(\GL_2\) on \(S_d\), that is given \(\alpha\in\GL_2\) and \(f\in S_d\),
\[
  \alpha\cdot f=\sym^d(\alpha)(f).
\] Observe that \(\sym^d(\alpha)\) is just the \(d\)-th graded component of the linear change of coordinates \(\sym(\alpha)\from S \to S\).

This action determines an action of the projective general linear group \(\PGL_2(\C)\), or simply \(\PGL_2\), on \(\proj S_d\).
Namely, given \(\varphi\in\PGL_2\), say the class of \(\alpha\in\GL_2\), we denote by \(\sym^d(\varphi)\) the class of \(\sym^d(\alpha)\in \GL_{d+1}\) in \(\PGL_{d+1}\) (which is clearly well defined).
So, given \(F\in \proj S_d\),
\[
  \varphi\cdot F=\sym^d(\varphi)(F).
\]

\begin{proposition}\label{prop:rk-2-and-GL2-orbits}
  The set of forms \(f\in S_d\) with Waring rank two is equal to the orbit \(\GL_2\cdot (x^d-y^d)\).
\end{proposition}
\begin{proof}
  Let \(f\in S_d\) be a form with Waring rank two.
  So, the form \(f\) decomposes as \(f=l_1^d-l_2^d\) and the linear forms \(l_1,l_2\in\proj S_1\) are non-proportional, or equivalently, pairwise \(\C\)-linearly independent.
  Hence, via a linear change of coordinates, we may assume \(l_1=x\) and \(l_2=y\).
\end{proof}

\begin{corollary}\label{coro:sq-freeness-of-rk-2}
  Every form in \(S\) with Waring rank two splits as a product of distinct linear forms.
\end{corollary}

\begin{remark}\label{rmk:any-number-of-variables}
  Given a form \(F\) in a standard graded polynomial ring \(\field[x_1,\dots,x_n]\) with coefficients in a field \(\field\), there is an analogue notion to Waring rank of \(F\) (see \cite[Definition 3.3]{carlini-grieve-oeding-4-lectures}).
  Notice that in this section we have not used so far that \(S\) is the polynomial ring over \(\C\) in two variables.
  \cref{prop:rk-2-and-GL2-orbits} and \cref{coro:sq-freeness-of-rk-2} hold for every field \(\field\) and integer \(d\) such that the Waring rank of \(x_1^d-x_2^d\in\field[x_1,\dots,x_n]\) is two, for every element \(a\in\field\) there is \(b\in\field\) with \(b^d=a\) and \(\field\) contains all the \(d\)-th roots of unity.
\end{remark}

\begin{definition}\label{def:RNC}
  We denote by \(\mathcal{C}\) the rational normal curve in \(\proj S_d\) corresponding to the image of the map that sends \(L\in\proj S_1\) to \(L^d\in \proj S_d\).
\end{definition}

\begin{definition}\label{def:tangent-and-2nd-sec}
  We denote by \(T \mathcal{C}\subseteq \proj S_d\) the union of all the lines tangent to \(\mathcal{C}\) and by \(\Sigma\subseteq \proj S_d\) the union of all the lines spanned by two different points of \(\mathcal{C}\).
\end{definition}

\begin{remark}\label{rmk:2dn-sec-2-rank-forms}
  The closure \(\overline{\Sigma}\) is the second secant variety of \(\mathcal{C}\) and the set \(\Sigma\mysetminus \mathcal{C}\) is the set of all forms \(F\in \proj S_d\) with Waring rank two.
\end{remark}

\begin{proposition}\label{prop:3-kind-point-on-secant-var-of-RNC}
  We may classify the points of \(\overline{\Sigma}\) into two different kinds.
  \begin{itemize}
  \item The \emph{tangent points}, points lying on \(T\mathcal{C}\).
  \item The \emph{secant points}, points \(p\) lying on some line spanned by two different points of \(\mathcal{C}\) but \(p\not\in\mathcal{C}\).
  \end{itemize}
\end{proposition}
\begin{proof}
  The first secant variety of \(\mathcal{C}\) is \(\mathcal{C}\) itself. Then, by \cite[Proposition 10.15]{Eisenbud-Harris-3264} a point in the subset \(\overline{\Sigma}\mysetminus \mathcal{C}\) of \(\overline{\Sigma}\) lies on the span of a unique divisor of degree 2 of \(\mathcal{C}\).
\end{proof}

\begin{remark}\label{rmk:properties-of-satrata-of-Sigma}
  Let \(p\) be a point of \(\overline{\Sigma}\).
  \begin{itemize}

  \item If \(p\) is a tangent point, there are \(L,N\in\proj S_1\) with \(p=NL^{d-1}\) and then its Waring rank is either \(d\) or \(1\).
    Furthermore, by \cref{thr:maximal-rank-forms}, the converse implication also holds.
  \item If \(p\) is a secant point, it corresponds to a form \(F\in\proj S_d\) with Waring rank 2 and then, by \cref{coro:sq-freeness-of-rk-2}, \(F\) is square free.
  \end{itemize}
\end{remark}

\begin{definition}\label{def:pi-and-pi}
  Given \(\mathscr{L}=(L_1,L_2,L_3) \in\ppp\), we denote by \(\Pi(\mathscr{L})\) the set of forms \(F\in \proj S_d\) that are multiple of \(L_1L_2L_3\).
  We denote by \(\Pi\) the particular case \(\mathscr{L}=(x+y,x,y)\).
\end{definition}

\begin{remark}\label{rmk:linearity-of-pi}
  Given \(\mathscr{L}=(L_1,L_2,L_3) \in\ppp\), consider \(\alpha_1,\alpha_2,\alpha_3\in \proj^1\) the respective roots of \(L_1,L_2,L_3\in\proj S_1\).
  Observe that \(\Pi(\mathscr{L})\subseteq \proj S_d\) is just the \emph{linear subspace} of \(\proj S_d\) corresponding to the forms \(F\in\proj S_d\) with roots at \(\alpha_1,\alpha_2,\alpha_2\in\proj^1\).
  In particular, the \emph{codimension} of \(\Pi(\mathscr{L})\subseteq\proj S_d\) is three.
\end{remark}

\begin{proposition}\label{prop:rank2-forms-on-Pi2}
  Let \(\mathscr{L}=(L_1,L_2,L_3)\in\ppp\).
  The Waring rank of every \(F\in\Pi(\mathscr{L})\cap \overline{\Sigma}\) is two.
\end{proposition}
\begin{proof}
  Since a form \(F\in \Pi(\mathscr{L})\cap \overline{\Sigma}\) is a multiple of three distinct linear forms, by \cref{rmk:properties-of-satrata-of-Sigma} the Waring rank of \(F\) is two.
\end{proof}

We denote by \(\Theta\subseteq \C\) the set of the \(d\)-th roots of the unity and we fix the usual anticlockwise order \(\{\xi_i\}_{i\in\nat_d}\) for \(\Theta\) where \(\xi_0=1\).

\begin{definition}\label{def:The-Gamma-map}
  Fix \(\mathscr{L}=(L_1,L_2,L_3) \in \ppp\).
  We define the map \(\Gamma_{\mathscr{L}}\from (\ttt)\to\Pi(\mathscr{L})\cap \overline{\Sigma}\) as follows.
  Given \(\xi=(\xi_i,\xi_j,\xi_k)\in \ttt\), there is a unique \(\varphi\in \PGL_2\) with
  \begin{align*}
    \varphi([x-\xi_iy]) &=L_1, & \varphi([x-\xi_jy]) &= L_2, & \varphi([x-\xi_ky]) &=L_3.
  \end{align*}
  Then we set \(\Gamma_{\mathscr{L}}(\xi)=\varphi\cdot [x^d-y^d]\).
  For the particular case \(\mathscr{L}=(x+y,x,y)\), we denote \(\Gamma_{\mathscr{L}}\) simply by \(\Gamma\).
\end{definition}

\begin{proposition}\label{prop:gamma-sujective}
  For all \(\mathscr{L}=(L_1,L_2,L_3) \in\ppp\), the map \(\Gamma_{\mathscr{L}}\) is surjective.
\end{proposition}
\begin{proof}
  Given \(F\in\Pi(\mathscr{L})\cap\overline{\Sigma}\), first by definition of \(\Pi(\mathscr{L})\), there is \(G\in\proj S_{d-3}\) with \(F=L_1L_2L_3G\) and, second by \cref{prop:rank2-forms-on-Pi2} and \cref{prop:rk-2-and-GL2-orbits}, there is \(\varphi\in\PGL_2\) with \(\varphi\cdot [x^d-y^d]=F\).

  So, for every \(i=1,2,3\) there is \(\xi_{j_i}\in\Theta\) with \(\varphi^{-1}(L_i)=[x-\xi_{j_i}y]\) and then \(\Gamma_{\mathscr{L}}(\xi_{j_1},\xi_{j_2},\xi_{j_3})=F\).
\end{proof}

\begin{corollary}\label{coro:finite-cardinal-of-Pi2}
  The cardinal of \(\Pi(\mathscr{L})\cap\overline{\Sigma}\) is at most \(d(d-1)(d-2)\), the cardinal of \(\ttt\).
  In particular, it is finite.
\end{corollary}

Let us describe the map \(\Gamma\) which has a particular neat description.
Fix \(\xi=(\xi_i,\xi_j,\xi_k)\in\ttt\).
Observe that the map \(\varphi\) in the \cref{def:The-Gamma-map} is a Möbius transformation.
Consider the affine chart \(U\subseteq\proj S_1\) corresponding to the coefficient of \(x\) being different from zero.
The map \(\varphi|_U\) sends respectively \(-\xi_i,-\xi_j,-\xi_k\) to \(1,0,\infty\).
So, for all \(l\in\nat_d\mysetminus\{k\}\)
\[
  \varphi([x-\xi_ly])=[x+(\xi_l,\xi_i;\xi_j,\xi_k)y],
\] where \((\xi_l,\xi_i;\xi_j,\xi_k)\) denotes the cross ratio of four complex numbers, and then
\[
  \Gamma(\xi)=\big[(x+y)xy\prod_{l\in\nat_d\mysetminus \{i,j,k\}} (x+(\xi_l,\xi_i;\xi_j,\xi_k)y)\big]\in\Pi\cap\overline{\Sigma}.
\] In this way, by \cref{prop:gamma-sujective}, we get an explicit description of \(\Pi\cap\overline{\Sigma}\),
\begin{align}\label{eq:R-0-via-triplets-of-integers}
  \Pi\cap\overline{\Sigma}=\im{\Gamma}\subseteq\proj S_d.
\end{align}

For example when \(d=4\), the \(4\)-th roots of the unity are the vertices of a square, hence their cross ratio is the harmonic ratio \(-1\), or one of its conjugates \(2\) or \(\frac{1}{2}\), and then
\[
  \Pi\cap\overline{\Sigma}=\{xy(x+y)(x-y),\ xy(x+y)(x+2y),\ xy(x+y)(x+\frac{1}{2}y)\}.
\] In this case we even know that the Waring rank of every \(F\in \Pi\mysetminus (\Pi\cap\overline{\Sigma})\) is \(3\), since clearly it is not \(1\) nor \(2\) and it is not \(4\) by \cref{thr:maximal-rank-forms}.
\medskip

We denote by \(\mathbf{D}\) the dihedral group of order \(2d\), which is the group of symmetries of \(\Theta\) (see~\cite[Section 1.2]{dummit-abstract-algebra}).
We consider \(\mathbf{D}\) acting componentwise on \(\ttt\).

\begin{lemma}\label{lem:cardinal-quotien-dihedral}
  The cardinal of \((\ttt)/\mathbf{D}\) is \(\binom{d-1}{2}\).
\end{lemma}
\begin{proof}
  It is a standard application of Burnside’s Lemma (see \cite[Exercice 8]{dummit-abstract-algebra}), but this case is particularly simple.
  Since every non-trivial \(\sigma\in \mathbf{D}\) acting on \(\Theta\) fixes at most two points, for all \(\xi\in\ttt\) the cardinals of the orbits \(\mathbf{D}\xi\) are equal to the cardinal of \(\mathbf{D}\).
  Hence,
  \[ \card{(\ttt)/\mathbf{D}}=\frac{\card{\ttt}}{\card{\mathbf{D}}}=\binom{d-1}{2}.\hfill\qedhere
  \]
\end{proof}

\begin{theorem}\label{thr:Gamma-dihedral-cover}
  For every \(\mathscr{L}\in\ppp\), the map \(\Gamma_{\mathscr{L}}\from \ttt\to \Pi(\mathscr{L})\cap\overline{\Sigma}\) is a \(\mathbf{D}\)-cover of \(\Pi(\mathscr{L})\cap\overline{\Sigma}\).
  In particular, the cardinal of the partition \((\ttt)/\Gamma_{\mathscr{L}}\) does not depend on \(\mathscr{L}\in\ppp\).
\end{theorem}

To prove \cref{thr:Gamma-dihedral-cover} we introduce two maps \(\mathbf{n}\), \(\mathbf{m}\) and we establish some of their elementary properties.
\medskip

Let \(\mathbb{S}^1\subseteq \C\) denote the unit circle.
For each \(\xi=(\xi_i,\xi_j,\xi_k) \in(\ttt)\), first we denote by \(\gamma_{\xi}\from \mathbb{S}^1\to \reals\cup\{\infty\}\) the map sending \(p\in\mathbb{S}^1\) to the cross ratio \(-(p,\xi_i;\xi_j,\xi_k)\in\reals\cup\{\infty\}\).
Second, we define \(\mathcal{S}^{\xi}_{i,j}\) as the unique connected component of the open subset \(\mathbb{S}^1\mysetminus \{\xi_i,\xi_j\}\subseteq\mathbb{S}^1\) such that \(\xi_k\not\in \mathcal{S}_{i,j}^{\xi}\). Similarly, we define the open subsets \(\mathcal{S}^{\xi}_{j,k}\subseteq \mathbb{S}^1\) and \(\mathcal{S}^{\xi}_{k,i}\subseteq\mathbb{S}^1\) of \(\mathbb{S}^1\).
Then, we define the maps \(\mathbf{n}\from \ttt\to \nat^3\) and \(\mathbf{m}\from \ttt\to \nat^3\) as sending \(\xi\in(\ttt)\) to
\begin{align*}
  \mathbf{n}(\xi)& =\Bigl(\card{\Theta\cap \mathcal{S}^{\xi}_{i,j}},~\card{\Theta\cap \mathcal{S}^{\xi}_{j,k}},~\card{\Theta\cap \mathcal{S}^{\xi}_{k,i}}\Bigr) \\
  \mathbf{m}(\xi)& =\Bigl(\card{\gamma_{\xi}(\Theta\cap \mathcal{S}^{\xi}_{i,j})},~\card{\gamma_{\xi}(\Theta\cap \mathcal{S}^{\xi}_{j,k})},~\card{\gamma_{\xi}(\Theta\cap \mathcal{S}^{\xi}_{k,i})}\Bigr).
\end{align*}

\begin{figure}
  \centering
  \begin{tikzpicture}
    \draw
    (0,0) coordinate (A)
    arc [start angle=0, end angle=60, radius=1] coordinate (a) 
    arc [start angle=60, end angle=120, radius=1] coordinate (B) 
    arc [start angle=120, end angle=180, radius=1] coordinate (b) 
    arc [start angle=180, end angle=240, radius=1] coordinate (C) 
    arc [start angle=240, end angle=300, radius=1] coordinate (c) 
    arc [start angle=300, end angle=360, radius=1]; 
    \filldraw[black] (A) circle (1pt) node[right] {$\xi_1$};
    \filldraw[black] (a) circle (1pt) node[xshift=2mm, yshift=2mm] {$\xi_2$};
    \filldraw[black] (B) circle (1pt) node[xshift=-2mm, yshift=2mm] {$\xi_3$};
    \filldraw[black] (b) circle (1pt) node[left] {$\xi_4$};
    \filldraw[black] (C) circle (1pt) node[xshift=-2mm, yshift=-2mm] {$\xi_5$};
    \filldraw[black] (c) circle (1pt) node[xshift=2mm, yshift=-2mm] {$\xi_6$};
  \end{tikzpicture}
  \caption{The set \(\Theta\subseteq \mathbb{S}^1\) for \(d=6\).}
  \label{pic-order-sixth-roots}
\end{figure}
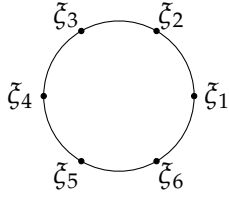

\begin{figure}
  \centering
  \begin{tikzpicture}
    \draw[color=red]
    (0,0) coordinate (A)
    arc [start angle=0, end angle=60, radius=1] coordinate (a) 
    arc [start angle=60, end angle=120, radius=1] coordinate (B); 

    \draw[color=blue] (B)
    arc [start angle=120, end angle=180, radius=1] coordinate (b) 
    arc [start angle=180, end angle=240, radius=1] coordinate (C); 
    \draw[color=green] (C)
    arc [start angle=240, end angle=300, radius=1] coordinate (c) 
    arc [start angle=300, end angle=360, radius=1]; 
    \filldraw[black] (A) circle (1pt) node[right] {$\xi_i$};
    \node[xshift=2mm, yshift=2mm] at (a) {$\mathcal{S}_{i,j}^{\xi}$};
    \filldraw[black] (B) circle (1pt) node[xshift=-2mm, yshift=2mm] {$\xi_j$};
    \node[left] at (b) {$\mathcal{S}_{j,k}^{\xi}$};
    \filldraw[black] (C) circle (1pt) node[xshift=-2mm, yshift=-2mm] {$\xi_k$};
    \node[xshift=2mm, yshift=-2mm] at (c) {$\mathcal{S}_{k,i}^{\xi}$};

    \draw[color=green]
    (6,0)
    arc [start angle=0, end angle=30, radius=1] coordinate (m1)
    arc [start angle=30, end angle=90, radius=1] coordinate (I); 
    \draw[color=blue] (I)
    arc [start angle=90, end angle=180, radius=1] coordinate (m2) 
    arc [start angle=180, end angle=210, radius=1] coordinate (Z); 
    \draw[color=red] (Z)
    arc [start angle=210, end angle=270, radius=1] coordinate (m3) 
    arc [start angle=270, end angle=330, radius=1] coordinate (U); 
    \draw[color=green] (U)
    arc [start angle=330, end angle=360, radius=1]; 

    \filldraw[black] (I) circle (1pt) node[above] {$\infty$};
    \filldraw[black] (Z) circle (1pt) node[xshift=-2mm, yshift=-2mm] {$0$};
    \filldraw[black] (U) circle (1pt) node[xshift=2mm, yshift=-2mm] {$-1$};
    \node[right] at (m1) {$\gamma_{\xi}(\mathcal{S}_{k,i}^{\xi})$};
    \node[left] at (m2) {$\gamma_{\xi}(\mathcal{S}_{j,k}^{\xi})$};
    \node[below] at (m3) {$\gamma_{\xi}(\mathcal{S}_{i,j}^{\xi})$};

    \draw[->, thick] (.75,.75) arc (130:90:2) coordinate (f) arc (90:50:2);
    \node[below] at (f) {$\gamma_{\xi}$};
    \node[above] at (f) {$\cong$};
  \end{tikzpicture}
  \caption{The map \(\gamma_{\xi}\) that sends resp. \(\xi_i,\xi_j,\xi_k\) to \(-1,0,\infty\).}
  \label{map-varphi}
\end{figure}
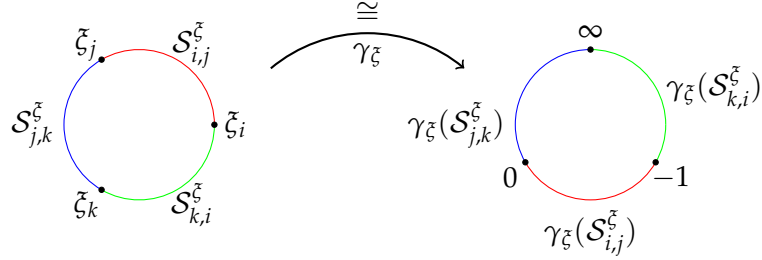

For example in \cref{pic-order-sixth-roots}, for \(\xi=(\xi_1,\xi_3,\xi_6)\) the value \(\mathbf{n}(\xi)\) is \((1,2,0)\) or for \(\xi=(\xi_6,\xi_3,\xi_1)\), it is \(\mathbf{n}(\xi)=(2,1,0)\).

\begin{lemma}\label{lem:about-m-and-n}
  Let \(\mathbf{n}\) and \(\mathbf{m}\) be the maps defined above.

  \begin{enumerate}
  \item\label{lem:about-m-and-n-it1} The maps \(\mathbf{n}\) and \(\mathbf{m}\) are equal.
  \item\label{lem:about-m-and-n-it2} The partition \((\ttt)/\Gamma\) is finer than \((\ttt)/\mathbf{m}\).
  \item\label{lem:about-m-and-n-it3} The cardinal of the image of \(\mathbf{n}\) is \(\binom{d-1}{2}\).
  \end{enumerate}
\end{lemma}

\begin{proof}
  \textit{1.} For all \(\xi\in\ttt\), the map \(\gamma_{\xi}\) is a Möbius transformation, so it is bijective.

  \textit{2.} For each \(\xi\in (\ttt)\), the set \(\gamma_{\xi}(\Theta\mysetminus\{\xi_k\})\subseteq \reals\) is the set of roots of \(\Gamma(\xi)(x,1)\).
  Hence, given \(\xi,\xi'\in (\ttt)\), if \(\Gamma(\xi)=\Gamma(\xi')\), then \(\mathbf{m}(\xi)=\mathbf{m}(\xi')\).

  \textit{3.} Observe that for all \(\xi\in(\ttt)\), the sum of the integers \(\mathbf{n}(\xi)\) is \(d-3\) and, for all triplets \((a,b,c)\in (\nat_{d-3})^3\) with \(a+b+c=d-3\), trivially \((a,b,c)=\mathbf{n}(\xi_0,\xi_{a+1},\xi_{(a+1)+(b+1)})\).
  So,
  \[
    \card{\im{\mathbf{n}}}=\card{\{(a,b,c)\in (\nat_{d-3})^3 \mbox{ : }a+b+c=d-3 \}}=\binom{d-1}{2}.\hfill \qedhere
  \]
\end{proof}

\begin{proof}[Proof of \cref{thr:Gamma-dihedral-cover}]
  Given \(\mathscr{L},\mathscr{L}'\in\ppp\) there is \(\varphi\in\PGL_2\) such that \(\Gamma_{\mathscr{L}}=\sym^d(\varphi)\circ \Gamma_{\mathscr{L}'}\).
  Since \(\sym^d(\varphi)\) is an isomorphism, the partition \((\ttt)/\Gamma_{\mathscr{L}}\) does not depend on \(\mathscr{L}\). So we restrict to the case \(\mathscr{L}=(x+y,x,y)\), that is \(\Gamma_{\mathscr{L}}=\Gamma\).
  \medskip

  Given \(\sigma\in \mathbf{D}\), the action of \(\sigma\) on \(\Theta\) is the restriction to \(\Theta\) of a symmetry of \(\C\), which is a Möbius transformation.
  Then, for all pairwise distinct \(\xi_l,\xi_i,\xi_j,\xi_k\in \Theta\),
  \[
    (\xi_l,\xi_i;\xi_j,\xi_k)=
    (\sigma(\xi_l),\sigma(\xi_i);\sigma(\xi_j),\sigma(\xi_k)).
  \] Hence the map \(\Gamma\) is invariant by the action of \(\mathbf{D}\) and then the partition \((\ttt)/\mathbf{D}\) is finer than \((\ttt)/\Gamma\).
  Finally, by \cref{lem:about-m-and-n}\eqref{lem:about-m-and-n-it2}, the partition \((\ttt)/\Gamma\) is finer than \((\ttt)/\mathbf{m}\) and, by \cref{lem:about-m-and-n}(\ref{lem:about-m-and-n-it1},\ref{lem:about-m-and-n-it3}) and \cref{lem:cardinal-quotien-dihedral}, the cardinals of the partitions \((\ttt)/\mathbf{D}\) and \((\ttt)/\mathbf{m}\) are equal and finite, hence the partitions \((\ttt)/\mathbf{D}\) and \((\ttt)/\Gamma\) are equal.
\end{proof}

\begin{proposition}\label{prop:cardinal-of-Pi2}
  The cardinal of \(\Pi(\mathscr{L})\cap\overline{\Sigma}\) is \(\binom{d-1}{2}\).
\end{proposition}
\begin{proof}
  By \cref{prop:gamma-sujective} the map \(\Gamma_{\mathscr{L}}\) is surjective, so the cardinals of \(\Pi(\mathscr{L})\cap\overline{\Sigma}\) and \((\ttt)/\Gamma_{\mathscr{L}}\) are equal.
  Then the claim follows from \cref{thr:Gamma-dihedral-cover} and \cref{lem:cardinal-quotien-dihedral}.
\end{proof}

\begin{proposition}\label{prop:set-of-maximal-gap-real-complex-rank}
  The set \(A\subseteq \proj \reals[x,y]_d\) of all forms with a maximal gap between their Waring and real Waring rank is equal to
  \[
    B=\bigcup_{\mathscr{L}\in(\proj \reals[x,y]_d)^3\mysetminus\Delta}\im{\Gamma_{\mathscr{L}}}
  \]
\end{proposition}
\begin{proof}
  Notice that the real Waring rank of a form in \(\reals[x,y]_d\) is one if and only if its Waring rank is one as well.
  Hence, by \cref{thr:maximal-possible-rank}, the maximal gap between the Waring and the real Waring rank of a form in \(\proj \reals[x,y]_d\) is at most \(d-2\).
  Clearly, by \cref{thr:hyperbolic-max-real-rank}, this possible maximal gap is reached by the forms in \(\Pi\cap\overline{\Sigma}\), so
  \[
    A=\{F\in\proj\reals[x,y]_d \mbox{ : } \rk(F)=2,\ \rk_{\reals}(F)=d\}
  \] and then, by \cref{prop:rk-2-and-GL2-orbits,thr:hyperbolic-max-real-rank}, \(A\subseteq B\).

  Given \(F\in B\), by \cref{prop:rank2-forms-on-Pi2}, the Waring rank of \(F\) is two.
  Hence, by \cref{thr:hyperbolic-max-real-rank}, if all the roots of \(F\) are real, then \(F\in A\).
  But given \(\mathscr{L}\in (\proj \reals[x,y]_d)^3\mysetminus\Delta\), via a linear change of coordinates of \(\proj\reals[x,y]_1\), we may assume \(\mathscr{L}=(x+y,x,y)\) and then, for every \(\xi\in \Theta^3\mysetminus \Delta\), the roots of \(\Gamma_{\mathscr{L}}(\xi)\) are real because they belong to the image of \(\gamma_{\xi}\) (see discussion below \cref{thr:Gamma-dihedral-cover}).
\end{proof}

The following theorem summarises the results of this section.

\begin{theorem}\label{thr:family-of-ls-generaly-intersecting-Sigma}
  For all \(\mathscr{L}=(L_1,L_2,L_3) \in\ppp\), there are exactly \(\binom{d-1}{2}\) distinct forms in \(\proj S_d\) with Waring rank two and being multiple of \(L_1L_2L_3\).

  The union \(\Pi(\mathscr{L})\cap\overline{\Sigma}\) for all \(\mathscr{L}\in\ppp\) is equal to \(\Sigma\mysetminus \mathcal{C}\), the set of all forms \(F\in\proj S_d\) with Waring rank two.
\end{theorem}
\begin{proof}
  Observe that every form \(F\in\proj S_d\) with Waring rank two and multiple of \(L_1L_2L_3\) belongs to \(\Pi(\mathscr{L})\cap\overline{\Sigma}\), so the first part follows from \cref{prop:cardinal-of-Pi2} and \cref{prop:rank2-forms-on-Pi2}.
  The second part follows from the definition of \(\Pi(\mathscr{L})\) and \cref{coro:sq-freeness-of-rk-2}.
\end{proof}

The main result of this section, \cref{thr:family-of-ls-generaly-intersecting-Sigma}, follows mainly from \cref{prop:cardinal-of-Pi2}.
To finish we offer another proof for the latter using classical results instead of the maps \(\Gamma_{\mathscr{L}}\).

First, we need \cref{lem:finite-number-of-rank2-forms}, which is already known, but we prove it without the maps \(\Gamma_{\mathscr{L}}\). 

\begin{lemma}\label{lem:finite-number-of-rank2-forms}
  Let \(\mathscr{L}=(L_1,L_2,L_3) \in \ppp\).
  The cardinal of \(\Pi(\mathscr{L})\cap\overline{\Sigma}\) is finite and it does not depend on the triplet \(\mathscr{L}\).
\end{lemma}
\begin{proof}
  By \cref{prop:rk-2-and-GL2-orbits} and \cref{prop:rank2-forms-on-Pi2}, every form \(F\in \Pi(\mathscr{L})\cap \overline{\Sigma}\) belongs to the orbit \(\PGL_2\cdot [x^d-y^d]\), but there are just finitely many \(\varphi\in \PGL_2\) with \(\varphi\cdot [x^d-y^d]\in \Pi(\mathscr{L})\).

  Given another triplet \(\mathscr{L}'=(L_1',L_2',L_3')\), there is \(\varphi\in \PGL_2\) such that \(\varphi(L_i)=L_i'\) for all \(i=1,2,3\).
  Then, \(\sym^d(\varphi)|_{\Pi(\mathscr{L})}\from \Pi(\mathscr{L})\to \Pi(\mathscr{L}')\) is an isomorphism which preserves the Waring rank.
  So, it determines a one-to-one correspondence between \(\Pi(\mathscr{L})\cap\overline{\Sigma}\) and \(\Pi(\mathscr{L}')\cap\overline{\Sigma}\).
\end{proof}

\begin{lemma}\label{lem:singularities-of-secant-var-of-RNC}
  The variety \(\overline{\Sigma}\) is smooth away from \(\mathcal{C}\).
\end{lemma}
\begin{proof}
  It is a particular case of some more general well-known results.
  See \cite{Coppens-sing-locus-sec-var-sm-pj-curve} or, for a more general setting, \cite[Theorem 1.1 or Proposition 2.3]{Vermeire-sing-of-secant-var} and \cite[Lemma 3.1]{Vermeire-reg-of-powers-of-ideal-sheaves}.
  For another setting see \cite[\textsection 8.15-8.16]{Graf-Hans-Hulek-geom-syzy-of-ellip-normal-curves-and-sec-var}.
  Recall \cite{Bertram-moduli-rk-2-vec-bundles}, a seminal work of all of them.
\end{proof}

\begin{proof}[Second proof of \cref{prop:cardinal-of-Pi2}]
  By \cref{lem:finite-number-of-rank2-forms}, we assume \(\mathscr{L}=(x+y,x,y)\).

  Since a form \(F\in\Pi\) has at least three distinct roots, by \cref{rmk:properties-of-satrata-of-Sigma}, \(F\not\in T \mathcal{C}\) and \(\Pi\cap \overline{\Sigma}\subseteq \left( \overline{\Sigma}\mysetminus T \mathcal{C} \right)\).
  Now we show that \(\Pi\) and \(\overline{\Sigma}\) intersect transversely.
  Let \(P\in\Pi\cap \overline{\Sigma}\).
  Since \(P\in \left(\overline{\Sigma}\mysetminus T \mathcal{C}\right)\), first by \cref{lem:singularities-of-secant-var-of-RNC}, \(P\) is a smooth point of \(\overline{\Sigma}\) and, second by \cref{rmk:properties-of-satrata-of-Sigma}, there are \(l,t\in S_1\) with \([l]\not=[t]\in \proj S_1\) such that \(P=[l^d+t^d]\in\proj S_d\).
  By Terracini's Lemma (see \cite{terracini-1911-sulle-vk} for the original statement, \cite[Lemma 2.15]{carlini-grieve-oeding-4-lectures} for a modern formulation and \cite[Lemma 2.1]{fujita-roberts-1981-var-small-sec-var} for a modern proof of a simplified version), the following morphism parametrises the projective tangent space of \(\overline{\Sigma}\) at \(P\)
  \[
    \begin{array}{cccc}
      \tau\ : & \proj (S_1\times S_1) & \to & \proj S_d \\
              & [n:m] & \to & [nl^{d-1}+mt^{d-1}].\\
    \end{array}
  \] By \cref{rmk:linearity-of-pi}, the projective tangent space of \(\Pi\) at \(P\) is \(\Pi\) itself.
  So, the varieties \(\overline{\Sigma}\) and \(\Pi\) intersect transversely at \(P\) if the pull back by \(\tau\) of the linear equations defining \(\Pi\) are linearly independent.
  Recall from \cref{rmk:linearity-of-pi} that \(\Pi\) is just the set of all forms with roots at \([1:-1],[0:1],[1:0]\in \proj^1\) (the roots of \(x+y,x,y\)).
  So, clearly such pull back just corresponds to the homogeneous linear equations on the coefficients of \([n:m]\in\proj(S_1\times S_1 )\) determined by the following matrix
  \[
    \left(
      \begin{array}{cccc}
        l^{d-1}(1,0) & 0 & t^{d-1}(1,0) & 0 \\
        0 & l^{d-1}(0,1) & 0 & t^{d-1}(0,1) \\
        l^{d-1}(1,-1) & -l^{d-1}(1,-1) & t^{d-1}(1,-1) & -t^{d-1}(1,-1) \\
      \end{array}
    \right).
  \] Using that \([l]\not=[t]\in\proj S_1\) and \(P=[l^d+t^d]\in \Pi\) we can see that the column rank is three.

  Finally, since \(\dim~\overline{\Sigma} + \dim~\Pi=(3) + (d-3)=d\) (cf. \cite[Proposition 11.32]{Harris-AG-first-course} and \cref{rmk:linearity-of-pi}), by \cite[Theorem 18.3]{Harris-AG-first-course} the intersection of \(\overline{\Sigma}\) and \(\Pi\) is a finite set with \(\deg~\overline{\Sigma}=\binom{d-1}{2}\) (cf. \cite[Theorem 10.16]{Eisenbud-Harris-3264}) points.
\end{proof}

\section{To the future}
\label{ssec:possible-generalisations}

\subsection{To higher rank}
\label{sssec:to-higher-rank}

\cref{prop:3-distinct-roots} below shows that the family of linear subspaces \(\Pi(\mathcal{L})\) contains all forms whose Waring rank is lower than the generic rank for degree \(d\) forms.

\begin{proposition}\label{prop:3-distinct-roots}
  Let \(F\in \proj S_d\mysetminus \mathcal{C}\) and denote \(r=\rk(F)\).
  If \(r\le\lfloor \frac{d+1}{2}\rfloor\), then \(F\) has at least three distinct roots.
\end{proposition}
\begin{proof}
  The form \(F\) has not a single root since \(F\not\in \mathcal{C}\).
  If \(F\) has just two distinct roots, via a linear change of coordinates we may assume \(F=x^ay^b\) with \(a+b=d\) and then \(\rk(F)=\max\{a,b\}+1> r\).
\end{proof}

So, it could be interesting intersect higher secant varieties of \(\mathcal{C}\) with the linear subspaces \(\Pi(\mathcal{L})\subseteq \proj S_d\).
One key point in order to describe such intersection for the rank two case is \cref{prop:3-kind-point-on-secant-var-of-RNC}, which, by \cite{comas-seiguer-rank-binary}, looks generalizable to
\[
  \sigma_r(\mathcal{C})\mysetminus \sigma_{r-1}(\mathcal{C})= (\mathcal{O}_{r-1}\mysetminus \mathcal{O}_{r-2})\sqcup (\sigma_r(\mathcal{C})\mysetminus \mathcal{O}_{r-1})
\] where \(\mathcal{O}_r\) is the \(r\)-th osculating variety of \(\mathcal{C}\) (i.e. \(\mathcal{O}_r=\overline{\cup_{P\in \mathcal{C}} \left\langle rP \right\rangle}\)).

\subsection{To higher dimensions}
\label{sssec:to-higher-dimensions}

When we consider \(S=\C[x_0,\dots,x_n]\), the dimension of \(\mathcal{C}\subseteq \proj S_d\) is \(n\) and the dimension of \(\Sigma\subseteq \proj S_d\) is \(2n+1\) (assuming \(d>4\), for \(d=4\) the cases \(n=3,4,5\) are respectively \(5,9,14\)-defectives).
Observe that \cref{prop:3-kind-point-on-secant-var-of-RNC} and \cref{rmk:properties-of-satrata-of-Sigma} are satisfied in any dimension.
That is, the forms belonging to \(\Sigma\) split as a product of linear forms.
So, denoting \(\proj L_d\) the natural image of \(S_1^{(d)}\) (the \(n\)-th symmetric power) into \(\proj S_d\), \(\Sigma\subseteq \proj L_d\).
Now, \(\dim \proj L_d=nd\) and the subset of forms in \(\proj L_d\) multiples of two given linear forms is a subvariety of codimension \(2n\).
Hence, we could try to intersect \(\Sigma\subseteq \proj L_d\) with subspaces of forms with a fixed root \(\alpha\in \proj^n\) and multiples of two linear forms \(l,t\in \proj S_1\) where \(\alpha\) is a root of neither \(l\) nor \(t\). Such varieties have the desired codimension in \(\proj L_d\).

\bibliographystyle{spmpsci}
\bibliography{general}

\begin{thebibliography}{10}
\providecommand{\url}[1]{{#1}}
\providecommand{\urlprefix}{URL }
\expandafter\ifx\csname urlstyle\endcsname\relax
  \providecommand{\doi}[1]{DOI~\discretionary{}{}{}#1}\else
  \providecommand{\doi}{DOI~\discretionary{}{}{}\begingroup
  \urlstyle{rm}\Url}\fi

\bibitem{AH95}
Alexander, J.E., Hirschowitz, A.: Polynomial interpolation in several
  variables.
\newblock J. Algebraic Geom. \textbf{4}(2), 201--222 (1995)

\bibitem{bernardi-gimigliano-ida-2011-computing-symmetric-rank}
Bernardi, A., Gimigliano, A., Id{\`a}, M.: Computing symmetric rank for
  symmetric tensors.
\newblock J. Symbolic Comput. \textbf{46}(1), 34--53 (2011).
\newblock \doi{10.1016/j.jsc.2010.08.001}.
\newblock \urlprefix\url{https://doi.org/10.1016/j.jsc.2010.08.001}

\bibitem{Bertram-moduli-rk-2-vec-bundles}
Bertram, A.: Moduli of rank-{$2$} vector bundles, theta divisors, and the
  geometry of curves in projective space.
\newblock J. Differential Geom. \textbf{35}(2), 429--469 (1992).
\newblock \urlprefix\url{http://projecteuclid.org/euclid.jdg/1214448083}

\bibitem{bialynicki-birula-2010-extrem-binar-forms}
Białynicki-Birula, A., Schinzel, A.: Extreme binary forms.
\newblock Acta Arithmetica \textbf{142}(3), 219--249 (2010).
\newblock \doi{10.4064/aa142-3-2}.
\newblock \urlprefix\url{https://doi.org/10.4064/aa142-3-2}

\bibitem{blekherman-2016-real-rank}
Blekherman, G., Sinn, R.: Real rank with respect to varieties.
\newblock Linear Algebra Appl. \textbf{505}, 344--360 (2016).
\newblock \doi{10.1016/j.laa.2016.04.035}.
\newblock \urlprefix\url{https://doi.org/10.1016/j.laa.2016.04.035}

\bibitem{BCG11}
Boij, M., Carlini, E., Geramita, A.V.: Monomials as sums of powers: the real
  binary case.
\newblock Proc. Amer. Math. Soc. \textbf{139}(9), 3039--3043 (2011).
\newblock \doi{10.1090/S0002-9939-2011-11018-9}.
\newblock \urlprefix\url{https://doi.org/10.1090/S0002-9939-2011-11018-9}

\bibitem{carlini-catalisano-2018-e-computability}
Carlini, E., Catalisano, M.V., Chiantini, L., Geramita, A.V., Woo, Y.:
  Symmetric tensors: rank, {S}trassen's conjecture and {$e$}-computability.
\newblock Ann. Sc. Norm. Super. Pisa Cl. Sci. (5) \textbf{18}(1), 363--390
  (2018)

\bibitem{CCG12}
Carlini, E., Catalisano, M.V., Geramita, A.V.: The solution to the {W}aring
  problem for monomials and the sum of coprime monomials.
\newblock J. Algebra \textbf{370}, 5--14 (2012).
\newblock \doi{10.1016/j.jalgebra.2012.07.028}.
\newblock \urlprefix\url{https://doi.org/10.1016/j.jalgebra.2012.07.028}

\bibitem{CCO17}
Carlini, E., Catalisano, M.V., Oneto, A.: Waring loci and the {S}trassen
  conjecture.
\newblock Adv. Math. \textbf{314}, 630--662 (2017).
\newblock \doi{10.1016/j.aim.2017.05.008}.
\newblock \urlprefix\url{https://doi.org/10.1016/j.aim.2017.05.008}

\bibitem{carlini-grieve-oeding-4-lectures}
Carlini, E., Grieve, N., Oeding, L.: Four lectures on secant varieties.
\newblock In: Connections between algebra, combinatorics, and geometry,
  \emph{Springer Proc. Math. Stat.}, vol.~76, pp. 101--146. Springer, New York
  (2014).
\newblock \doi{10.1007/978-1-4939-0626-0_2}.
\newblock \urlprefix\url{https://doi.org/10.1007/978-1-4939-0626-0_2}

\bibitem{causa-re-2011-maximum-real-rank}
Causa, A., Re, R.: On the maximum rank of a real binary form.
\newblock Ann. Mat. Pura Appl. (4) \textbf{190}(1), 55--59 (2011).
\newblock \doi{10.1007/s10231-010-0137-2}.
\newblock \urlprefix\url{https://doi.org/10.1007/s10231-010-0137-2}

\bibitem{comas-seiguer-rank-binary}
Comas, G., Seiguer, M.: On the rank of a binary form.
\newblock Found. Comput. Math. \textbf{11}(1), 65--78 (2011).
\newblock \urlprefix\url{https://doi.org/10.1007/s10208-010-9077-x}

\bibitem{comon-2008-symmet-tensor}
Comon, P., Golub, G., Lim, L.H., Mourrain, B.: Symmetric tensors and symmetric
  tensor rank.
\newblock SIAM Journal on Matrix Analysis and Applications \textbf{30}(3),
  1254--1279 (2008).
\newblock \doi{10.1137/060661569}.
\newblock \urlprefix\url{https://doi.org/10.1137/060661569}

\bibitem{comon-ottaviani-2012-typical-real-rank}
Comon, P., Ottaviani, G.: On the typical rank of real binary forms.
\newblock Linear Multilinear Algebra \textbf{60}(6), 657--667 (2012).
\newblock \doi{10.1080/03081087.2011.624097}.
\newblock \urlprefix\url{https://doi.org/10.1080/03081087.2011.624097}

\bibitem{Coppens-sing-locus-sec-var-sm-pj-curve}
Coppens, M.: The singular locus of the secant varieties of a smooth projective
  curve.
\newblock Arch. Math. (Basel) \textbf{82}(1), 16--22 (2004).
\newblock \doi{10.1007/s00013-003-4772-3}.
\newblock \urlprefix\url{http://dx.doi.org/10.1007/s00013-003-4772-3}

\bibitem{dummit-abstract-algebra}
Dummit, D.S., Foote, R.M.: Abstract algebra, third edn.
\newblock John Wiley \& Sons, Inc., Hoboken, NJ (2004)

\bibitem{Eisenbud-Harris-3264}
Eisenbud, D., Harris, J.: 3264 and all that---a second course in algebraic
  geometry.
\newblock Cambridge University Press, Cambridge (2016).
\newblock \doi{10.1017/CBO9781139062046}.
\newblock \urlprefix\url{http://dx.doi.org/10.1017/CBO9781139062046}

\bibitem{fujita-roberts-1981-var-small-sec-var}
Fujita, T., Roberts, J.: Varieties with small secant varieties: the extremal
  case.
\newblock Amer. J. Math. \textbf{103}(5), 953--976 (1981).
\newblock \doi{10.2307/2374254}.
\newblock \urlprefix\url{https://doi.org/10.2307/2374254}

\bibitem{Ger96}
Geramita, A.V.: Inverse systems of fat points: {W}aring's problem, secant
  varieties of {V}eronese varieties and parameter spaces for {G}orenstein
  ideals.
\newblock In: The {C}urves {S}eminar at {Q}ueen's, {V}ol. {X} ({K}ingston,
  {ON}, 1995), \emph{Queen's Papers in Pure and Appl. Math.}, vol. 102, pp.
  2--114. Queen's Univ., Kingston, ON (1996)

\bibitem{young-alg-of-invariants}
Grace, J.H., Young, A.: The algebra of invariants.
\newblock Cambridge Library Collection. Cambridge University Press, Cambridge
  (2010).
\newblock \doi{10.1017/CBO9780511708534}.
\newblock \urlprefix\url{https://doi.org/10.1017/CBO9780511708534}.
\newblock Reprint of the 1903 original

\bibitem{Graf-Hans-Hulek-geom-syzy-of-ellip-normal-curves-and-sec-var}
Graf V.~Bothmer, H.C., Hulek, K.: Geometric syzygies of elliptic normal curves
  and their secant varieties.
\newblock Manuscripta Math. \textbf{113}(1), 35--68 (2004).
\newblock \doi{10.1007/s00229-003-0421-1}.
\newblock \urlprefix\url{http://dx.doi.org/10.1007/s00229-003-0421-1}

\bibitem{Harris-AG-first-course}
Harris, J.: Algebraic geometry, \emph{Graduate Texts in Mathematics}, vol. 133.
\newblock Springer-Verlag, New York (1992).
\newblock \doi{10.1007/978-1-4757-2189-8}.
\newblock \urlprefix\url{http://dx.doi.org/10.1007/978-1-4757-2189-8}.
\newblock A first course

\bibitem{iarrobino-determinantal-loci}
Iarrobino, A., Kanev, V.: Power sums, {G}orenstein algebras, and determinantal
  loci, \emph{Lecture Notes in Mathematics}, vol. 1721.
\newblock Springer-Verlag, Berlin (1999).
\newblock \doi{10.1007/BFb0093426}.
\newblock \urlprefix\url{https://doi.org/10.1007/BFb0093426}.
\newblock Appendix C by Iarrobino and Steven L. Kleiman

\bibitem{kolda-2009-tensor-decom-applic}
Kolda, T.G., Bader, B.W.: Tensor decompositions and applications.
\newblock SIAM Review \textbf{51}(3), 455--500 (2009).
\newblock \doi{10.1137/07070111x}.
\newblock \urlprefix\url{https://doi.org/10.1137/07070111x}

\bibitem{kung-1986-gundelfinger-bin-forms}
Kung, J.P.S.: Gundelfinger's theorem on binary forms.
\newblock Stud. Appl. Math. \textbf{75}(2), 163--169 (1986).
\newblock \doi{10.1002/sapm1986752163}.
\newblock \urlprefix\url{https://doi.org/10.1002/sapm1986752163}

\bibitem{kung-1987-canonical-bin-forms-even-deg}
Kung, J.P.S.: Canonical forms for binary forms of even degree.
\newblock In: Invariant theory, \emph{Lecture Notes in Math.}, vol. 1278, pp.
  52--61. Springer, Berlin (1987).
\newblock \doi{10.1007/BFb0078806}.
\newblock \urlprefix\url{https://doi.org/10.1007/BFb0078806}

\bibitem{kung-1990-canonical-bin-forms-theme-sylvester}
Kung, J.P.S.: Canonical forms of binary forms: variations on a theme of
  {S}ylvester.
\newblock In: Invariant theory and tableaux ({M}inneapolis, {MN}, 1988),
  \emph{IMA Vol. Math. Appl.}, vol.~19, pp. 46--58. Springer, New York (1990)

\bibitem{kung-rota-1984-invariant-theory-binary-forms}
Kung, J.P.S., Rota, G.C.: The invariant theory of binary forms.
\newblock Bull. Amer. Math. Soc. (N.S.) \textbf{10}(1), 27--85 (1984).
\newblock \doi{10.1090/S0273-0979-1984-15188-7}.
\newblock \urlprefix\url{https://doi.org/10.1090/S0273-0979-1984-15188-7}

\bibitem{Lan12}
Landsberg, J.M.: Tensors: geometry and applications, \emph{Graduate Studies in
  Mathematics}, vol. 128.
\newblock American Mathematical Society, Providence, RI (2012)

\bibitem{reznick-1996-homogeneous-poly-to-pde}
Reznick, B.: Homogeneous polynomial solutions to constant coefficient {PDE}'s.
\newblock Adv. Math. \textbf{117}(2), 179--192 (1996).
\newblock \doi{10.1006/aima.1996.0007}.
\newblock \urlprefix\url{https://doi.org/10.1006/aima.1996.0007}

\bibitem{reznick-2012-length-binary-forms}
Reznick, B.: On the length of binary forms.
\newblock In: Quadratic and higher degree forms, \emph{Dev. Math.}, vol.~31,
  pp. 207--232. Springer, New York (2013).
\newblock \doi{10.1007/978-1-4614-7488-3_8}.
\newblock \urlprefix\url{https://doi.org/10.1007/978-1-4614-7488-3_8}

\bibitem{sylvester-1851-essay-forms-sketch-eli-trans-can-forms}
Sylvester, J.J.: An essay on canonical forms, supplement to a sketch of a
  memoir on elimination, transformation and canonical forms  (1851).
\newblock Paper 34 in {\it Collected Mathematical Papers of James Joseph
  Sylvester}, vol. I, 203--216, Chelsea, New York, 1973, originally published
  by Cambridge University Press, London, Fetter Lane, E.C. 1904

\bibitem{sylvester-1851-remarkable-dicovery-forms-and-hyperdeterminants}
Sylvester, J.J.: On a remarkable discovery in the theory of canonical forms and
  of hyperdeterminants.
\newblock The London, Edinburgh, and Dublin Philosophical Magazine and Journal
  of Science \textbf{2}(12), 391--410 (1851).
\newblock \doi{10.1080/14786445108645733}.
\newblock \urlprefix\url{https://doi.org/10.1080/14786445108645733}.
\newblock Paper 41 in {\it Collected Mathematical Papers of James Joseph
  Sylvester}, vol. I, 265--283, Chelsea, New York, 1973, originally published
  by Cambridge University Press, London, Fetter Lane, E.C. 1904

\bibitem{sylvester-65-Newtons-hitherto-rule}
Sylvester, J.J.: On an elementary proof and generalization of {S}ir {I}saac
  {N}ewton's hitherto undemonstrated rule for the discovery of imaginary roots.
\newblock Proc.London Math.Soc. \textbf{1}, 1--16 (1865).
\newblock Paper 84 in {\it Collected Mathematical Papers of James Joseph
  Sylvester}, vol. II, 498--513, Chelsea, New York, 1973, originally published
  by Cambridge University Press, London, Fetter Lane, E.C. 1904--1912

\bibitem{Tei15}
Teitler, Z.: Sufficient conditions for {S}trassen's additivity conjecture.
\newblock Illinois J. Math. \textbf{59}(4), 1071--1085 (2015).
\newblock \urlprefix\url{http://projecteuclid.org/euclid.ijm/1488186021}

\bibitem{terracini-1911-sulle-vk}
Terracini, A.: Sulle $v_k$ per cui la variet{\`a} degli $s_h\,(h + 1)$ seganti
  ha dimensione minore dell'ordinario.
\newblock Rendiconti del Circolo matematico di Palermo (1884-1940)
  \textbf{31}(1), 392--396 (1911).
\newblock \doi{10.1007/bf03018812}.
\newblock \urlprefix\url{https://doi.org/10.1007/bf03018812}

\bibitem{tokcan-2017-warin-rank}
Tokcan, N.: On the {W}aring rank of binary forms.
\newblock Linear Algebra Appl. \textbf{524}, 250--262 (2017).
\newblock \doi{10.1016/j.laa.2017.03.007}.
\newblock \urlprefix\url{https://doi.org/10.1016/j.laa.2017.03.007}

\bibitem{Vermeire-reg-of-powers-of-ideal-sheaves}
Vermeire, P.: On the regularity of powers of ideal sheaves.
\newblock Compositio Math. \textbf{131}(2), 161--172 (2002).
\newblock \doi{10.1023/A:1014913511483}.
\newblock \urlprefix\url{http://dx.doi.org/10.1023/A:1014913511483}

\bibitem{Vermeire-sing-of-secant-var}
Vermeire, P.: Singularities of the secant variety.
\newblock J. Pure Appl. Algebra \textbf{213}(6), 1129--1132 (2009).
\newblock \doi{10.1016/j.jpaa.2008.11.006}.
\newblock \urlprefix\url{http://dx.doi.org/10.1016/j.jpaa.2008.11.006}

\end{thebibliography}

\end{document}